\documentclass[11pt]{article}
\usepackage[T1]{fontenc}
\usepackage[utf8]{inputenc}
\usepackage{geometry}
\usepackage{amsmath}
\usepackage{amsthm}
\usepackage{setspace}
\usepackage[unicode=true,
 bookmarks=false,
 breaklinks=false,pdfborder={0 0 1},colorlinks=true,allcolors=black]
 {hyperref}
\usepackage{comment}
\makeatletter

\theoremstyle{plain}
\newtheorem{thm}{\protect\theoremname}[section]
\theoremstyle{plain}
\newtheorem{lem}[thm]{\protect\lemmaname}
\theoremstyle{plain}

\theoremstyle{plain}
\newtheorem{cor}[thm]{\protect\corname}
\theoremstyle{plain}
\newtheorem{fact}[thm]{\protect\factname}
\theoremstyle{plain}
\newtheorem{conj}[thm]{\protect\conjecturename}
\numberwithin{equation}{section}

\usepackage{amssymb}

\allowdisplaybreaks

\makeatother

\providecommand{\claimname}{Claim}
\providecommand{\corname}{Corollary}
\providecommand{\lemmaname}{Lemma}
\providecommand{\theoremname}{Theorem}
\providecommand{\conjecturename}{Conjecture}
\providecommand{\factname}{Fact}

\newcommand{\su}{\subseteq}
\newcommand{\eps}{\varepsilon}

\title{On Graham's rearrangement conjecture}
\author{Huy Tuan Pham\thanks{Department of Mathematics, California Institute for Technology, Pasadena, CA 91125. Email: {\tt htpham@caltech.edu}. Research supported by a Clay Research Fellowship.} \and Lisa Sauermann\thanks{Institute for Applied Mathematics, University of Bonn, Germany. Email: {\tt sauermann@iam.uni-bonn.de}. Research supported by the DFG Heisenberg Program.}}
\begin{document}

\maketitle

\abstract{Graham conjectured in 1971 that for any prime $p$, any subset $S\subseteq \mathbb{Z}_p\setminus \{0\}$ admits an ordering $s_1,s_2,\dots,s_{|S|}$ where all partial sums $s_1, s_1+s_2,\dots,s_1+s_2+\dots+s_{|S|}$ are distinct. We prove this conjecture for all subsets $S\subseteq \mathbb{Z}_p\setminus \{0\}$ with $|S|\le p^{1-\alpha}$ and $|S|$ sufficiently large with respect to $\alpha$, for any $\alpha \in (0,1)$. Combined with earlier results, this gives a complete resolution of Graham's rearrangement conjecture for all sufficiently large primes $p$.}

\section{Introduction}

Given an abelian group $G$ and a subset $S\su G$, a \emph{valid ordering} of $S$ is an ordering $s_1,s_2,\dots,s_{|S|}$ of the elements of $S$ such that all partial sums $s_1,s_1+s_2,\dots,s_1+s_2+\dots+s_{|S|}$ are distinct. %

Valid orderings have been a subject of a number of works in combinatorics and group theory. The main subject of this paper is the a conjecture, called Graham's rearrangement conjecture, which was posed by Graham \cite[p. 36]{G} in 1971  and later reiterated by Erd\H{o}s and Graham \cite[p. 95]{EG}. 
\begin{conj}
    Any subset $S\su \mathbb{Z}_p\setminus \{0\}$ admits a valid ordering. 
\end{conj}

Graham's rearrangement conjecture has attracted a lot of interest over the past few years. Bedert and Kravitz \cite{BK} showed that the conjecture holds true for small sets $S$, namely for $|S| \le \exp((\log p)^{1/4})$, improving an earlier result of Kravitz \cite{Kr} and an observation of Sawin \cite{Sa}. M\"uyesser and Pokrovskiy \cite{MP} showed that the conjecture holds for very large $S$ of size $|S|\ge (1-o(1))p$. Buci\'c, Frederickson, M\"uyesser, Pokrovskiy and Yepremyan \cite{BFMPY} showed an approximate version of the conjecture, where all but $o(|S|)$ many partial sums are guaranteed to be distinct. Recently, Bedert, Buci\'c, Kravitz, Montgomery and M\"uyesser \cite{BBKMM} showed that the conjecture holds true for large sets $S$, namely for $|S| \ge p^{1-c}$ for a small positive constant $c$. 

The main result of this paper is the following theorem, which shows that Graham's rearrangement conjecture holds for all $S\su \mathbb{Z}_p\setminus \{0\}$ with $C_\alpha\le |S|\le p^{1-\alpha}$, for any $\alpha>0$ (with a suitable constant $C_\alpha>0$ depending on $\alpha$).

\begin{thm}\label{thm:main}
    For any $0<\alpha<1$, there exists a constant $C_\alpha>0$ such that the following holds. Let $p$ be a prime and let $S\subseteq \mathbb{Z}_p \setminus \{0\}$ be a subset with $C_\alpha\le |S|\le p^{1-\alpha}$. Then there exists a valid ordering of $S$, i.e.\ there is an ordering $s_1,s_2,\dots,s_{|S|}$ of the elements of $S$ such that all partial sums $s_1,s_1+s_2,\dots,s_1+s_2+\dots+s_{|S|}$ are distinct.  
\end{thm}

Together with the earlier results discussed above, this completely settles Graham's rearrangement conjecture for all sufficiently large primes $p$.

Unlike previous results towards Graham's rearrangement conjecture, employing structural additive combinatorial results, or graph-theoretic arguments relying on suitable notions of expansion and the absorption method, our proof is inherently probabilistic, relying on anticoncentration estimates on sums of random subsets of $S$ of a given size. More precisely, we prove the following anticoncentration result, which is an important input for our proof of Theorem \ref{thm:main}.

\begin{thm}\label{thm:anticonc}
    There exists an absolute constant $C>0$ such that the following holds. Let $p$ be a prime, let $S\subseteq \mathbb{Z}_p$ be a subset of size $|S|\ge 2$, and let $m$ be an integer with $C\log |S|\le m\le 10^{-3}|S|/\log |S|$. Now, let $R\su S$ be a uniformly random subset of $S$ of size $m$, and let  $\Sigma(R) = \sum_{x\in R}x\in \mathbb{Z}_p$ be the sum of the elements of $R$. Then 
    \[
        \max_{z\in \mathbb{Z}_p} \mathbb{P}[\Sigma(R)=z] \le \frac{1}{p} + \frac{C}{|S|\sqrt{m}}.
    \]
\end{thm}

In our proof of Theorem \ref{thm:anticonc}, we represent a uniformly random subset $R$ of size $m$ by first partitioning $S$ randomly into $m$ parts of roughly equal size, and then choose a uniformly random element from each part. To understand the anticoncentration, given a fixed partition, we employ Fourier analytic ideas similar to the work of Nguyen and Vu \cite{NV} on the inverse  Littlewood--Offord problem. The key new step in the proof is the use of suitable concentration estimates to relate certain quantities involving the random partition of $S$ to their deterministic analogs. 

In order to also handle the case where $m$ is larger than the upper bound assumed in Theorem \ref{thm:anticonc}, we use the following corollary of Theorem \ref{thm:anticonc}.

\begin{cor}\label{coro:anticonc}
    For any $0<\eps<1$ There exists a constant $C'_\eps>0$ such that the following holds. Let $p$ be a prime, let $S\subseteq \mathbb{Z}_p$ be a subset of size $|S|\ge 2$, and consider a positive integer $m\le (1-\eps)|S|$. Now, let $R\su S$ be a uniformly random subset of $S$ of size $m$, and let  $\Sigma(R) = \sum_{x\in R}x\in \mathbb{Z}_p$ be the sum of the elements of $R$. Then 
    \[
        \max_{z\in \mathbb{Z}_p} \mathbb{P}[\Sigma(R)=z] \le \frac{1}{p} + \frac{C'_\eps\sqrt{\log |S|}}{|S|\sqrt{m}}.
    \]
\end{cor}

Roughly speaking, our proof of Theorem \ref{thm:main} proceeds as follows. Starting from a random ordering of $S$, we perform local adjustments for every zero-sum segment (flipping the endpoint of the segment with another later element nearby in the ordering). Using the anticoncentration bounds in Theorem \ref{thm:anticonc} and Corollary \ref{coro:anticonc}, we show that this way we can obtain a valid ordering of $S$.

Beside the setup of the cyclic group in Graham's rearrangement conjecture, valid orderings have also been studied in other groups. Alspach \cite{AL, CP} conjectured an analog of Graham's rearrangement conjecture in arbitrary finite abelian groups. We believe that our argument may generalize to other abelian groups as well, and plan to return Alspach's conjecture in future work. %

\paragraph{Notation.} For a prime $p$, and any subset $S\su \mathbb{Z}_p$, we write $\Sigma(S) = \sum_{x\in S}x\in \mathbb{Z}_p$ for the sum of the elements in $S$. All logarithms are to base $e$, unless specified otherwise.

\paragraph{Organization.} After some preliminaries in Section \ref{sec:prelim}, we will prove our anticoncentration result, Theorem \ref{thm:anticonc}, in Section \ref{sec:anticonc}. In Section \ref{sec:anticonc-de}, we deduce some corollaries of the anticoncentration result, in particular Corollary \ref{coro:anticonc} and an anticoncentration result for a chain of random subsets of given sizes. Finally, we give the proof of out main result, Theorem \ref{thm:main}, in Section \ref{sec:rearrangement}, using the anticoncentration results in Section \ref{sec:anticonc-de}. 

\section{Preliminaries}\label{sec:prelim}

For $y\in\mathbb{R}$, let us define $\|y\|_\mathbb{Z} := \min_{z\in \mathbb{Z}} |y-z|$ to be the distance of $y$ to the closest integer (and note that $\|y\|_\mathbb{Z}\le 1$ for all $y\in\mathbb{R}$).

\begin{fact}\label{fact1}
For all $y_1,\dots,y_k\in\mathbb{R}$, we have
\[\|y_1+\dots+y_k\|_\mathbb{Z}^2\le k\cdot (\|y_1\|_\mathbb{Z}^2+\dots+\|y_k\|_\mathbb{Z}^2).\]
\end{fact}
\begin{proof}
    Let $z_i\in \mathbb{Z}$ be such that $\|y_i\|_{\mathbb{Z}} = |y_i-z_i|$. We then have 
    \[
        \|y_1+\dots +y_k\|_{\mathbb{Z}} \le |(y_1+\dots +y_k)-(z_1+\dots +z_k)|=\Bigg|\sum_{i=1}^{k} (y_i-z_i)\Bigg|\le \sum_{i=1}^{k} |y_i-z_i|= \sum_{i=1}^{k} \|y_i\|_{\mathbb{Z}}.
    \]
    By Cauchy-Schwarz inequality, we can then conclude that 
    \[
        \|y_1+\dots +y_k\|_{\mathbb{Z}}^2 \le \Bigg(\sum_{i=1}^{k} \|y_i\|_{\mathbb{Z}}\Bigg)^2 \le k \cdot (\|y_1\|_\mathbb{Z}^2+\dots+\|y_k\|_\mathbb{Z}^2). \qedhere
    \]
\end{proof}

\begin{fact}\label{fact2}
For any $y\in\mathbb{R}$, we have
\[ 1-20\|y\|_\mathbb{Z}^2 \le \cos(2\pi y)\le 1 - 2\|y\|_\mathbb{Z}^2.\]
\end{fact}
\begin{proof}
    Note that all of the terms in the inequality are $1$-periodic functions of $y$, which are furthermore symmetric around $0$. It therefore suffices to check the inequality for $y\in[0,1/2]$.
    
    For $y\in[0,1/2]$, we have $\|y\|_\mathbb{Z}=y$, and therefore the desired inequality simplifies to
    \[1-20y^2 \le \cos(2\pi y)\le 1-2y^2.\]
    By Taylor's theorem (with the Lagrange form of the remainder), for any $y\in[0,1/2]$, we have
    \[\cos(2\pi y)=1-\frac{(2\pi)^2}{2}y^2+\frac{(2\pi)^3\sin(2\pi \xi)}{6}y^3\]
    for some $\xi\in [0,y]\su [0,1/2]$. Since $\sin(2\pi \xi)\ge 0$ and $y\in [0,1/2]$, we can conclude
    \[\cos(2\pi y)\ge 1-\frac{(2\pi)^2}{2}y^2= 1-(2\pi^2)y^2\ge 1-20y^2.\]
    Similarly, we have
    \[
    \cos(2\pi y)=1-\frac{(2\pi)^2}{2}y^2+\frac{(2\pi)^4\cos(2\pi \xi)}{24}y^4,
    \]
    for some $\xi \in [0,y]$. Using that $\cos(2\pi \xi)\le 1$ and $y\in [0,1/2]$, we have
    \[\cos(2\pi y)\le 1-\frac{(2\pi)^2}{2}y^2 + \frac{(2\pi)^4}{24} \cdot \Big(\frac{1}{2}\Big)^2 y^2 \le  1-2y^2.\qedhere\]
\end{proof}

For a prime $p$, and $y\in\mathbb{Z}$, we furthermore define $\|y\|_p:=\|y/p\|_\mathbb{Z} = \min_{z\in \mathbb{Z}} |y/p-z|$. Noting that $\|y\|_p$ is $p$-periodic, this allows us to define $\|x\|_p$ for any $x\in\mathbb{Z}_p$ by setting $\|x\|_p:=\|y\|_p$ for any representative $y\in\mathbb{Z}$ of the residue class $x\in\mathbb{Z}_p$.

\begin{fact}\label{fact3}
For a prime $p$, and $x_1,\dots,x_k\in\mathbb{Z}_p$, we have
\[\|x_1+\dots+x_k\|_p^2\le k\cdot (\|x_1\|_p^2+\dots+\|x_k\|_p^2).\]
\end{fact}
\begin{proof}
    Choosing representatives $y_1,\dots,y_k\in\mathbb{Z}$ for $x_1,\dots,x_k\in\mathbb{Z}_p$, it follows from Fact \ref{fact3} that
    \[\|x_1+\dots+x_k\|_p^2=\|(y_1/p)+\dots+(y_k/p)\|_\mathbb{Z}^2\le k\cdot (\|y_1/p\|_\mathbb{Z}^2+\dots+\|y_k/p\|_\mathbb{Z}^2)=k\cdot (\|x_1\|_p^2+\dots+\|x_k\|_p^2).\qedhere\]
\end{proof}

Finally, we record an easy consequence of the Cauchy-Davenport theorem.

\begin{fact}\label{fact:cauchy-davenport}
    For a prime $p$, a subset $A\su \mathbb{Z}_p$, and a positive integer $k$, let us consider the $k$-fold sumset $kA=A+\dots+A=\{a_1+\dots+a_k:a_1,\dots,a_k\in A\}$. If $kA\ne \mathbb{Z}_p$, we have
    \[|kA|\ge 1+k\cdot(|A|-1).\]
\end{fact}
\begin{proof}
Recall that the Cauchy-Davenport theorem states that for any subsets $A,B\su \mathbb{Z}_p$ with $A+B\ne \mathbb{Z}_p$, we have $|A+B|\ge |A|+|B|-1$, and hence $|A+B|-1\ge (|A|-1)+(|B|-1)$. Applying this repeatedly, for any subsets $A_1,\dots,A_k$ with $A_1+\dots+A_k\ne \mathbb{Z}_p$ (and hence in particular $A_1+\dots+A_j\ne \mathbb{Z}_p$ for $j=1,\dots,k$) we can conclude
\[|A_1+\dots+A_k|-1\ge (|A_1+\dots+A_{k-1}|-1)+(|A_k|-1)\ge \dots\ge (|A_1|-1)+\dots+(|A_k|-1)\]
Taking $A_1=\dots=A_k=A$ yields $|kA|-1\ge k\cdot(|A|-1)$, which implies the desired inequality.
\end{proof}

Finally, for a prime $p$ and $y\in\mathbb{Z}$, we define $e_p(y)=\exp(2\pi i y/p)$. Noting that this function is $p$-periodic, we can now defined $e_p(x)$ for any $x\in \mathbb{Z}_p$ by setting $e_p(x)=e_p(y)=\exp(2\pi i y/p)$ for any representative $y\in\mathbb{Z}$ of the residue class $x\in\mathbb{Z}_p$.

\begin{fact}\label{fact4}
    For a prime $p$, and $x\in \mathbb{Z}_p$, we have
    \[\operatorname{Re}(e_p(x))\le 1-2\|x\|_p^2.\]
\end{fact}
\begin{proof}
    Letting $y\in \mathbb{Z}$ be a representative of the residue class $x\in\mathbb{Z}_p$, by Fact \ref{fact2} we have
    \[\operatorname{Re}(e_p(x))=\operatorname{Re}(e_p(y))=\operatorname{Re}(\exp(2\pi i y/p))=\cos(2\pi y/p)\le 1-2\|y/p\|_\mathbb{Z}^2=1-2\|y\|_p^2=1-2\|x\|_p^2.\qedhere\]
\end{proof}

\section{Anticoncentration on boolean slices via Fourier Analysis}\label{sec:anticonc}

In this section, we prove Theorem \ref{thm:main}, taking $C=2^{24}$. So let $p$ be a prime, and $S\su \mathbb{Z}_p$ a subset of size $|S|\ge 2$. Furthermore, let $m$ be a positive integers with $2^{24}\log |S|=C\log |S|\le m\le 10^{-3}|S|/\log |S|$. Note that this in particular implies $|S|\ge m\ge 2^{24}\ge 10^7$.

Let $\widehat{\mathbb{Z}_p}$ denote the set of Fourier characters over $\mathbb{Z}_p$ (recall that these are the group homomorphisms $\chi:\mathbb{Z}_p\to \{z\in \mathbb{C}: |z|=1\}$). Also recall that we can identify a character $\chi:\mathbb{Z}_p\to \{z\in \mathbb{C}: |z|=1\}$ with an element $\chi\in \mathbb{Z}_p$ via 
\[
    \chi(x) = e_p(\chi x)
\]
(recall that in Section \ref{sec:prelim}, for any $x\in \mathbb{Z}_p$ we defined $e_p(x)=\exp(2\pi i y/p)$ where $y\in \mathbb{Z}$ is a representative of the residue class $x\in \mathbb{Z}_p$).

Given the set $S\subseteq \mathbb{Z}_p$, in the statement of Theorem \ref{thm:main}  we consider a uniformly random subset $R\su S$ of size $m$. We can sample such a uniformly random size-$m$ subset in the following way, which will be helpful in our proof. %
We first randomly partition the elements of $S$ into $m$ sets $S_1,\dots,S_m$ with $|S_1|\ge |S_2|\ge \dots\ge |S_m|\ge |S_m|-1$ (this means that $S_1,\dots,S_m$ each have size $\lfloor |S|/m\rfloor+1$ or $\lfloor |S|/m\rfloor$, and exactly the first $|S|-m\lfloor |S|/m\rfloor$ of the sets $S_1,\dots,S_m$ have size $\lceil |S|/m\rceil+1$). Since $m\le |S|/4$, all of the sets $S_1,\dots,S_m$ have size at least $4$. We then independently pick exactly one element $X_i$ from each set $S_i$ to form a set $R = \{X_1,\dots,X_m\}$. It is easy to see that $R$ is distributed precisely as a uniformly random subset of $S$ of size $m$. Also note that  $\Sigma(R)=X_1+\dots+X_m$. We will write $\mathcal{S}=(S_1,\dots,S_m)$ for the $m$-tuple of sets in the partition of $S$, and $X=(X_1,\dots,X_m)$ for the $m$-tuple of elements chosen from these sets. %

We write $\mathbb{P}_\mathcal{S}$ and $\mathbb{E}_\mathcal{S}$ for probabilities and expectations over the randomness of $\mathcal{S}=(S_1,\dots,S_m)$, i.e. over the randomness of the partition $S=S_1\cup\dots\cup S_m$. For a fixed outcome of $\mathcal{S}=(S_1,\dots,S_m)$, we write $\mathbb{P}_X$ and $\mathbb{E}_X$ to denote probabilities and expectations over the randomness of $X=(X_1,\dots,X_n)$, conditional on $\mathcal{S}=(S_1,\dots,S_m)$. Note that conditional on $\mathcal{S}=(S_1,\dots,S_m)$, the elements $X_i\in S_i$ are independent uniformly random elements of $S_i$ for $i=1,\dots,m$.

Now, we have
\[
    \mathbb{P}[\Sigma(R)=z] = \mathbb{E}_\mathcal{S} [\mathbb{P}_X[\Sigma(R) = z]],
\]
where the probability $\mathbb{P}_X[\Sigma(R) = z]$ is conditional on $\mathcal{S}=(S_1,\dots,S_m)$, over the randomness of $X=(X_1,\dots,X_n)$.

Let us fix an outcome of $\mathcal{S}=(S_1,\dots,S_m)$. 
The standard identity
\begin{equation*}%
    \mathbb{P}_X[\Sigma(R) = z] =\mathbb{P}_X[X_1+\dots+X_m =z]= \frac{1}{p} \mathbb{E}_X\!\Bigg[\sum_{\chi \in \widehat{\mathbb{Z}_p}} \!\chi(X_1+\dots+X_m-z) \Bigg]\!= \frac{1}{p}\sum_{\chi \in \widehat{\mathbb{Z}_p}} \!\chi(-z)\prod_{i=1}^{m} \mathbb{E}_X[\chi(X_i)]
\end{equation*}
gives
\begin{equation}\label{eq:fourier-expand}
    \mathbb{P}_X[\Sigma(R) = z] \le \frac{1}{p}\sum_{\chi \in \widehat{\mathbb{Z}_p}} \bigg|\chi(-z)\prod_{i=1}^{m} \mathbb{E}[\chi(X_i)]\bigg|=\frac{1}{p}\sum_{\chi \in \widehat{\mathbb{Z}_p}} \ \prod_{i=1}^{m} |\mathbb{E}_X[\chi(X_i)]|.
\end{equation}

For any $\chi \in \widehat{\mathbb{Z}_p}$ and any $i=1,\dots,m$, we have
\[|\mathbb{E}_X[\chi(X_i)]| = \Bigg| \frac{1}{|S_i|} \sum_{x\in S_i} \chi(X_i) \Bigg|= \Bigg| \frac{1}{|S_i|}\sum_{x\in S_i} e_p(\chi x) \Bigg|=  \Bigg(\frac{1}{|S_i|^2}\sum_{x\in S_i} e_p(\chi x) \cdot \overline{\sum_{x\in S_i} e_p(\chi x)}\Bigg)^{1/2}\]
As $\sum_{x\in S_i} e_p(\chi x) \cdot \overline{\sum_{x\in S_i} e_p(\chi x)}\in \mathbb{R}$ and $\overline{e_p(\chi x)}=e_p(-\chi x)$ for all $x\in \mathbb{Z}_p$, we can observe that
\begin{align*}
\frac{1}{|S_i|^2}\sum_{x\in S_i} e_p(\chi x) \cdot \overline{\sum_{x\in S_i} e_p(\chi x)}&=\frac{1}{|S_i|^2}\operatorname{Re}\Bigg(\sum_{x\in S_i} e_p(\chi x) \cdot \sum_{x\in S_i} \overline{e_p(\chi x)}\Bigg)\\
&=\frac{1}{|S_i|^2}\operatorname{Re}\Bigg(\sum_{x, x'\in S_i}  e_p(\chi x-\chi x')\Bigg)\\
&\le\frac{1}{|S_i|^2}\sum_{x, x'\in S_i}  (1-2\|\chi x-\chi x'\|_p^2)\\
&=1-\frac{2}{|S_i|^2}\sum_{x, x'\in S_i}  \|\chi x-\chi x'\|_p^2,
\end{align*}
where the inequality is due to Fact \ref{fact4}. This yields
\[|\mathbb{E}_X[\chi(X_i)]|\le \Bigg(1-\frac{2}{|S_i|^2}\sum_{x, x'\in S_i}  \|\chi x-\chi x'\|_p^2\Bigg)^{1/2}\le \exp\Bigg(-\frac{1}{|S_i|^2}\sum_{x, x'\in S_i}  \|\chi x-\chi x'\|_p^2\Bigg).\]
Thus, we obtain
\begin{equation}\label{eq:bound-with-psi-for-one-chi}
\prod_{i=1}^{m} |\mathbb{E}_X[\chi(X_i)]|\le \exp\Bigg(-\sum_{i=1}^{m}\frac{1}{|S_i|^2}\sum_{x, x'\in S_i} \|\chi x-\chi x'\|_p^2\Bigg)=\exp(-\psi(\chi)),
\end{equation}
where we define
\[
    \psi(\chi) := \sum_{i=1}^{m} \frac{1}{|S_i|^2} \sum_{x,x'\in S_i} \|\chi x-\chi x'\|_p^2
\]
for any $\chi\in  \mathbb{Z}_p$.

Note that the value of $\psi(\chi)$ depends on the outcome of $\mathcal{S}=(S_1,\dots,S_m)$, and that for the trivial character corresponding to the element $\chi=0\in \mathbb{Z}_p$, we always have $\psi(\chi)=0$. Also note that for all $\chi\in  \mathbb{Z}_p$, we have
\begin{equation}\label{eq:ineq-for-psi}
    \psi(\chi) \ge \frac{m^2}{2|S|^2}\sum_{i=1}^{m} \sum_{x,x'\in S_i} \|\chi x-\chi x'\|_p^2,
\end{equation}
since for $i=1,\dots,m$ we have $|S_i|\le \lfloor |S|/m\rfloor+1\le \sqrt{2}|S|/m$. Finally, since $\|\chi x-\chi x'\|_p^2\le 1$ for all $x,x'\in S$, we always have $\psi(\chi)\le m$.

Now, let
\[
    A_0 := \Bigg\{\chi\in  \mathbb{Z}_p  : \psi(\chi) \in [0,1) \Bigg\},
\]
and for any positive integer $t$, let
\[
    A_t := \Bigg\{\chi\in  \mathbb{Z}_p  : \psi(\chi) \in [t,2t) \Bigg\}.
\]
Then we obtain a partition $\mathbb{Z}_p=A_0\cup A_1\cup A_2\cup A_4\cup A_8\cup \dots \cup A_{2^{\lfloor \log_2 m\rfloor}}$ (noting that $A_t=\emptyset$ for $t>m$). Furthermore, by \eqref{eq:fourier-expand} and \eqref{eq:bound-with-psi-for-one-chi} we have
\[
   \mathbb{P}_X[\Sigma(R) = z] \le \frac{1}{p}\sum_{\chi \in \widehat{\mathbb{Z}_p}} \ \prod_{i=1}^{m} |\mathbb{E}_X[\chi(X_i)]|  \le \frac{1}{p}\sum_{\chi \in \mathbb{Z}_p}\exp(-\psi(\chi)) \le \frac{1}{p}|A_0|+\frac{1}{p}\sum_{\ell= 0}^{\lfloor \log_2 m\rfloor} |A_{2^{\ell}}| \exp(-2^{\ell}).
\]
Henceforth, we obtain
\begin{align}
    \mathbb{P}[\Sigma(R)=z] &= \mathbb{E}_\mathcal{S} [\mathbb{P}_X[\Sigma(R) = z]]\notag \\
    &\le \mathbb{E}_\mathcal{S}\Bigg[\frac{1}{p}|A_0|+\frac{1}{p}\sum_{\ell= 0}^{\lfloor \log_2 m\rfloor} |A_{2^{\ell}}| \exp(-2^{\ell})\Bigg] \notag\\
    &= \frac{1}{p} \mathbb{E}_\mathcal{S}[|A_0|]+\frac{1}{p}\sum_{\ell= 0}^{\lfloor \log_2 m\rfloor} \mathbb{E}_\mathcal{S}[|A_{2^{\ell}}|] \cdot \exp(-2^{\ell}). \label{eq:prob}
\end{align}

In the remaining part of the section, we will bound $\mathbb{E}_\mathcal{S}[|A_{t}|]$ for $t \in \{0,1,2,4,8,\dots\}$. For $\chi\in \mathbb{Z}_p$, let us  define
\[
    \Psi(\chi) =  \frac{m}{|S|^2}\sum_{x,x'\in S} \|\chi x-\chi x'\|_p^2.
\]
For any positive integer $t$, let
\[
    B_t = \{\chi \in \mathbb{Z}_p : \Psi(\chi) \le t\}.
\]
Furthermore, let $D_t$ be the collection of $\chi \in \mathbb{Z}_p\setminus\{0\}$ such that there is some $y\in \mathbb{Z}_p$ with
\[\Big|\Big\{x\in S: \|\chi x-y\|_p\le 8\sqrt{t/m}\Big\}\Big|\ge \frac{3}{4}|S|\]
(in other words, this means that there is some ``interval'' in $\mathbb{Z}_p$ of length $16p\sqrt{t/m}$ containing at least a $(3/4)$-fraction of the elements of $\chi S$).

Note that, for any positive integer $t$, the sets $B_t$ and $D_t$ do not depend on the outcome of $\mathcal{S}=(S_1,\dots,S_m)$ (they only depend on the given set $S\su \mathbb{Z}_p$). Our next aim is to relate the sets $A_t$, which do depend on the outcome of $\mathcal{S}=(S_1,\dots,S_m)$, to $B_t$ and $D_t$. More specifically, we will show that any $\chi\in \mathbb{Z}_p$ with $\chi\not\in D_t$ is contained in $A_t$ with small probability. Similarly, we will show any $\chi\in D_t\setminus B_{32t}$ is contained in $A_t$ with small probability. We will then obtain the desired bound for $\mathbb{E}_\mathcal{S}[|A_t|]$ by upper-bounding $|B_{32t}|$. 

\begin{lem}\label{lem:est1}
    Let $t$ be a positive integer, and let $\chi \in \mathbb{Z}_p\setminus\{0\}$  with $\chi\not\in D_t$. Then we have
    \[
        \mathbb{P}_\mathcal{S}[\psi(\chi) < 2t] \le \frac{1}{|S|^9}. 
    \]
\end{lem}
\begin{proof}
For an element $x'\in S$, let $i(x')\in \{1,\dots,m\}$ denote the index such that $x'\in S_{i(x')}$, i.e. the index of the part of the random partition $S=S_1\cup\dots\cup S_m$ containing $x'$. Since $\chi\not\in D_t$, for every fixed $x'\in S$ we have $|\{x\in S: \|\chi x-\chi x'\|_p> 8\sqrt{t/m}\}|\ge |S|/4$. Note that, when conditioning on $i(x')$ (i.e. exposing and conditioning on which part of the random partition the element $x'$ belongs to), the set $S_{i(x')}\setminus \{x'\}$ is a uniformly random subset of $S\setminus \{x'\}$ of size either $\lfloor |S|/m\rfloor$ (if $i(x')\le |S|-m \lfloor |S|/m\rfloor$) or $\lfloor |S|/m\rfloor-1$ (if $i(x')> |S|-m \lfloor |S|/m\rfloor$). Thus, conditioning on $i(x')$, the set $S_{i(x')}\setminus \{x'\}$ is a uniformly random subset of $S\setminus \{x'\}$ of some given size $k\ge |S|/(2m)$. Recalling that there are at least $|S|/4$ elements $x\in S\setminus \{x'\}$ with $\|\chi x-\chi x'\|_p> 8\sqrt{t/m}$, by the Chernoff bound for hypergeometric distributions (see \cite[Theorem 2.10 and Eq. (2.6)]{random-graphs-book}), with probability at least $1-e^{-k/32}\ge 1-e^{-|S|/(64m)}$ there are at least $k/8\ge |S|/(16m)$ different elements $x\in S_{i(x')}\setminus \{x'\}$ with $\|\chi x-\chi x'\|_p> 8\sqrt{t/m}$.

Thus, by a union bound, with probability at least $1-|S|e^{-|S|/(64m)}$ for all $x'\in S$ we have at least $|S|/(16m)$ different elements $x\in S_{i(x')}\setminus \{x'\}$ with $\|\chi x-\chi x'\|_p> 8\sqrt{t/m}$. Whenever this happens, by \eqref{eq:ineq-for-psi} we obtain
\[
\psi(\chi)\ge \frac{m^2}{2|S|^2} \sum_{i=1}^{m} \sum_{x,x'\in S_i} \|\chi x-\chi x'\|_p^2
        = \frac{m^2}{2|S|^2}   \sum_{x'\in S}  \sum_{x\in S_{i(x')}} \|\chi x-\chi x'\|_p^2 
        \ge \frac{m^2}{2|S|^2}  \cdot |S|\cdot  \frac{|S|}{16m} \cdot \big(8\sqrt{t/m}\big)^2=2t.
\]
Thus, we can conclude
\[\mathbb{P}_\mathcal{S}[\psi(\chi) < 2t]\le |S|\cdot e^{-|S|/(64m)}\le |S|\cdot e^{-10\log |S|}=\frac{1}{|S|^9},\]
recalling our assumption $m\le 10^{-3}|S|/\log |S|$.
\end{proof}

Our next goal is to show that any $\chi\in D_t\setminus B_{32t}$ is contained in $A_t$ with small probability. For any positive integer $t$, and any $\chi\in D_t$, let $y_\chi\in \mathbb{Z}_p$ be such that $|\{x\in S: \|\chi x-y_\chi\|_p\le 8\sqrt{t/m}\}|\ge \frac{3}{4}|S|$ holds, i.e. such that at least a $(3/4)$-fraction of the elements of $\chi S$ are contained in the ``interval'' in $\mathbb{Z}_p$ of length $16p\sqrt{t/m}$ around $y_\chi$. Note that we can choose $y_\chi$ to be independent of $t$ by considering the smallest positive integer $t$ with $\chi\in D_t$ (choosing a valid $y_\chi\in \mathbb{Z}_p$ for this smallest value of $t$ will automatically give a valid $y_\chi$ for all larger values of $t$). Now, let
\[J_{\chi,t}=\{x\in \mathbb{Z}_p : \|\chi x-y_\chi\|_p\le 16\sqrt{t/m}\}.\]

\begin{lem}\label{lem:est-Psi}
    For any positive integer $t$, and any $\chi \in D_t\setminus B_{2000t}$, we have 
    \[
        \sum_{x\in S\setminus J_{\chi,t}} \|\chi x-y_\chi\|_p^2 \ge \frac{200t}{m}\cdot |S|.
    \]
\end{lem}
\begin{proof}
    By Fact \ref{fact3} we have
    \[
        \|\chi x-\chi x'\|_p^2 \le 2 \|\chi x-y_\chi\|_p^2 + 2 \|\chi x'-y_\chi\|_p^2,%
    \]
    from which we obtain  (recalling that $\chi\not\in B_{2000t}$ and hence $\Psi(\chi)\ge 2000t$)
    \[\frac{2000t}{m}\le\frac{\Psi(\chi)}{m} = \frac{1}{|S|^2}\sum_{x,x'\in S} \|\chi x-\chi x'\|_p^2\le \frac{2}{|S|^2}\sum_{x,x'\in S} (\|\chi x-y_\chi\|_p^2 + \|\chi x'-y_\chi\|_p^2)
        = \frac{4}{|S|} \sum_{x\in S} \|\chi x-y_\chi\|_p^2.\]
        Noting that
        \[\sum_{x\in S} \|\chi x-y_\chi\|_p^2=\sum_{x\in S\setminus J_{\chi,t}} \|\chi x-y_\chi\|_p^2+\sum_{x\in S\cap J_{\chi,t}} \|\chi x-y_\chi\|_p^2\le \sum_{x\in S\setminus J_{\chi,t}} \|\chi x-y_\chi\|_p^2+|S\cap J_{\chi,t}|\cdot \frac{265t}{m},\]
        we can conclude that
        \[\frac{2000t}{m}\le\frac{4}{|S|}\sum_{x\in S\setminus J_{\chi,t}} \|\chi x-y_\chi\|_p^2+\frac{4|S\cap J_{\chi,t}|}{|S|}\cdot \frac{265t}{m}\le \frac{4}{|S|}\sum_{x\in S\setminus J_{\chi,t}} \|\chi x-y_\chi\|_p^2+\frac{1024t}{m}.\]
        Thus, we obtain
        \[\sum_{x\in S\setminus J_{\chi,t}} \|\chi x-y_\chi\|_p^2 \ge \frac{|S|}{4}\cdot \frac{976t}{m}\ge \frac{200t}{m}\cdot |S|.\qedhere\]
        \end{proof}

\begin{lem}\label{lem:est2}
    For any positive integer $t$, and any $\chi \in D_t\setminus B_{2000t}$, we have 
    \[
        \mathbb{P}_\mathcal{S}[\psi(\chi) < 2t] \le \frac{1}{|S|^9}. 
    \]
\end{lem}
\begin{proof}
    For an element $x\in S$, let us again write $i(x)\in\{1,\dots,m\}$ for the index such that $x\in S_{i(x)}$. Recall that $|\{x'\in S: \|\chi x'-y_\chi\|_p\le 8\sqrt{t/m}\}|\ge \frac{3}{4}|S|$, so at least $3|S|/4$ elements $x'\in S$ satisfy $\|\chi x'-y_\chi\|_p\le 8\sqrt{t/m}$.

    Now, for any fixed $x\in S\setminus J_{\chi,t}$ (i.e. any $x\in S$ with $\|\chi x-y_\chi\|_p> 16\sqrt{t/m}$), when conditioning on $i(x)$, the set $S_{i(x)}\setminus \{x\}$ is a uniformly random subset of $S\setminus \{x\}$ of size either $\lfloor |S|/m\rfloor$ (if $i(x)\le |S|-m \lfloor |S|/m\rfloor$) or $\lfloor |S|/m\rfloor-1$ (if $i(x)> |S|-m \lfloor |S|/m\rfloor$). Thus, conditioning on $i(x)$, the set $S_{i(x)}\setminus \{x\}$ is a uniformly random subset of $S\setminus \{x\}$ of some given size $k\ge |S|/(2m)$. Recalling that there are at least $3|S|/4$ elements $x'\in S\setminus \{x\}$ with $\|\chi x'-y_\chi\|_p\le 8\sqrt{t/m}$, by the Chernoff bound for hypergeometric distributions (see \cite[Theorem 2.10 and Eq. (2.6)]{random-graphs-book}), with probability at least $1-e^{-k/24}\ge 1-e^{-|S|/(48m)}$ there are at least $k/2\ge |S|/(4m)$ different elements $x'\in S_{i(x)}\setminus \{x\}$ with $\|\chi x'-y_\chi\|_p\le 8\sqrt{t/m}\le \|\chi x-y_\chi\|_p/2$. These elements $x'\in S_{i(x)}\setminus \{x\}$ then satisfy $\|\chi x-\chi x'\|_p\ge \|\chi x-y_\chi\|_p-\|\chi x'-y_\chi\|_p\ge \|\chi x-y_\chi\|_p/2$

    Thus, by a union bound, with probability at least $1-|S|e^{-|S|/(48m)}$ for all $x\in S\setminus J_{\chi,t}$ we have at least $|S|/(4m)$ different elements $x'\in S_{i(x)}\setminus \{x\}$ with $\|\chi x-\chi x'\|_p\ge  \|\chi x-y_\chi\|_p/2$.  Whenever this happens, by \eqref{eq:ineq-for-psi} we obtain
\begin{align*}
\psi(\chi)&\ge \frac{m^2}{2|S|^2} \sum_{i=1}^{m} \sum_{x,x'\in S_i} \|\chi x-\chi x'\|_p^2\\
        &= \frac{m^2}{2|S|^2}  \cdot \sum_{x\in S}  \sum_{x'\in S_{i(x)}} \|\chi x-\chi x'\|_p^2 \\
        &\ge \frac{m^2}{2|S|^2}  \cdot \sum_{x\in S\setminus J_{\chi,t}}  \sum_{x'\in S_{i(x)}} \|\chi x-\chi x'\|_p^2\\
        &\ge \frac{m^2}{2|S|^2}  \cdot \sum_{x\in S\setminus J_{\chi,t}}  \frac{|S|}{4m}\cdot (\|\chi x-y_\chi\|_p/2)^2\\
        &= \frac{m}{32|S|}  \cdot \sum_{x\in S\setminus J_{\chi,t}}  \|\chi x-y_\chi\|_p^2\ge \frac{m}{32|S|}  \cdot\frac{200t}{m}\cdot |S|\ge 5t,
\end{align*}
where for the penultimate inequality we used Lemma \ref{lem:est-Psi}. Thus, we can conclude
\[\mathbb{P}_\mathcal{S}[\psi(\chi) < 2t]\le \mathbb{P}_\mathcal{S}[\psi(\chi) < 5t]\le |S|\cdot e^{-|S|/(48m)}\le |S|\cdot e^{-10\log |S|}=\frac{1}{|S|^9},\]
recalling our assumption $m\le 10^{-3}|S|/\log |S|$.
\end{proof}

Our final ingredient for the proof of Theorem \ref{thm:anticonc} is the following upper bound for $|B_t|$. 
\begin{lem}\label{lem:tildeC}
    For any positive integer $t\le m/2000$, we have
    \[
        |B_t|\le 1+\frac{200p\sqrt{t}}{|S|\sqrt{m}}.
    \]
\end{lem}

We will need some auxiliary lemmas for the proof of Lemma \ref{lem:tildeC}. For $\delta>0$, define 
\[
    Q_{t,\delta} = \Bigg\{x \in \mathbb{Z}_p:  \sum_{\chi \in B_t} \|\chi x\|_p^2 < \delta |B_t|\Bigg\}.
\]

\begin{lem}\label{lem:lower-bound-Q}
    For every positive integer $t$, we have
    \[|Q_{t, 10t/m}|\ge \frac{9}{10}|S|.\]
\end{lem}

\begin{proof}
Let $Y,Y'\in S$ be two independent uniformly random elements of $S$. Note that then for any $\chi\in\mathbb{Z}_p$, we have
\[\mathbb{E}_{Y,Y'\in S}\Big[\|\chi Y-\chi Y'\|_p^2\Big]= \frac{1}{|S|^2}\sum_{x,x'\in S} \|\chi x-\chi x'\|_p^2=\frac{\Psi(\chi)}{m}.\]
Thus, we obtain
\[\mathbb{E}_{Y,Y'\in S}\Bigg[\sum_{\chi \in B_t}\|\chi Y-\chi Y'\|_p^2\Bigg]=\sum_{\chi \in B_t} \frac{\Psi(\chi)}{m}\le |B_t|\cdot \frac{t}{m}.\]
Therefore, by Markov's inequality, we can conclude
\[\mathbb{P}_{Y,Y'\in S}\Big[Y-Y' \not\in Q_{t, 10t/m}\Big]=\mathbb{P}_{Y,Y'\in S}\Bigg[\sum_{\chi \in B_t}\|\chi Y-\chi Y'\|_p^2\ge\frac{10t}{m}|B_t|\Bigg]\le \frac{|B_t|\cdot \frac{t}{m}}{\frac{10t}{m}|B_t|}=\frac{1}{10}.\]

Thus, with probability at least $9/10$, we have $Y-Y'\in Q_{t, 10t/m}$. In particular, this means that there exists some fixed $y'\in S$ such that for uniformly random $Y\in S$ with probability at least $9/10$, we have $Y-y'\in Q_{t, 10t/m}$. This means that there are at least $(9/10)|S|$ elements $y\in S$ with $y-y'\in Q_{t, 10t/m}$. Noting that the resulting elements $y-y'\in Q_{t, 10t/m}$ are distinct for all these $y\in S$, we can conclude that $|Q_{t, 10t/m}|\ge (9/10)|S|$.
\end{proof}

\begin{lem}\label{lem:dual-est}
    For every positive integer $t$, we have 
    \[
        |Q_{t,1/200}| \le \frac{5}{4}\cdot \frac{p}{|B_t|}. 
    \]
\end{lem}
\begin{proof}
Note that $\Psi(0)=0$ and $\Psi(\chi)=\Psi(-\chi)$ for all $\chi\in \mathbb{Z}_p$. Therefore we have $0\in B_t$, and $-\chi\in B_t$ for all $\chi\in B_t$. Hence, we obtain
\[\sum_{x\in \mathbb{Z}_p} \Bigg(\sum_{\chi\in B_t} e_p(\chi x)\Bigg)^2=\sum_{x\in \mathbb{Z}_p}\Bigg(\sum_{\chi\in B_t} e_p(\chi x)\Bigg)\cdot \Bigg(\sum_{\chi'\in B_t} e_p(\chi x)\Bigg)=\sum_{\chi,\chi'\in B_t} \sum_{x\in \mathbb{Z}_p} e_p((\chi+\chi') x)=|B_t|\cdot p.\]

On the other hand, by Fact \ref{fact2}, for any $z\in \mathbb{Z}$, we have
\[\frac{e_p(z)+e_p(-z)}{2}=\frac{\exp(2\pi iz/p)+\exp(-2\pi iz/p)}{2}=\cos(2\pi z/p)\ge 1 - 20 \|z/p\|_\mathbb{Z}^2=1 - 20 \|z\|_p^2.\]
Therefore, for any $x\in Q_{t,1/200}$, we have
\[
    \sum_{\chi\in B_t} e_p(\chi x)=\sum_{\chi\in B_t} \frac{e_p(\chi x)+e_p(-\chi x)}{2} \ge \sum_{\chi\in B_t}(1 - 20 \|\chi x\|_p^2)= |B_t| - 20 \sum_{\chi\in B_t}\|\chi x\|_p^2\ge |B_t| - 20 \cdot \frac{|B_t|}{200},
\]
and hence
\[
    \Bigg(\sum_{\chi\in B_t} e_p(\chi x)\Bigg)^2 \ge \Bigg( |B_t| - 20 \cdot \frac{|B_t|}{200}\Bigg)^2= \Bigg(\frac{9}{10}|B_t| \Bigg)^2= \frac{81}{100}|B_t|^2\ge \frac{4}{5}|B_t|^2. 
\]

Thus, we can conclude
\[p\cdot |B_t|=\sum_{x\in \mathbb{Z}_p} \Bigg(\sum_{\chi\in B_t} e_p(\chi x)\Bigg)^2\ge \sum_{x\in Q_{t,1/200}} \Bigg(\sum_{\chi\in B_t} e_p(\chi x)\Bigg)^2\ge |Q_{t,1/200}|\cdot \frac{4}{5}|B_t|^2.\]
This yields
\[
    |Q_{t,1/200}| \le \frac{5}{4}\cdot \frac{p}{|B_t|}. \qedhere
\]
\end{proof}

\begin{lem}\label{lem:sumset}
    Let $t$ be a positive integer, and let $\delta>0$. For some positive integer $k$, we consider the $k$-fold sum $kQ_{t,\delta}=Q_{t,\delta}+\dots+Q_{t,\delta}=\{x_1+\dots+x_k: x_1,\dots,x_k\in Q_{t,\delta}\}$. Then we have $kQ_{t,\delta} \subseteq Q_{t, k^2 \delta}$.
\end{lem}
\begin{proof}
    For $x_1,\dots,x_k\in Q_{t,\delta}$, by Fact \ref{fact3} we have
        \[
        \|\chi (x_1+\dots+x_k)\|_p^2 =\|\chi x_1+\dots+\chi x_k\|_p^2\le k\cdot (\|\chi x_1\|_p^2+\dots+\|\chi x_k\|_p^2)
    \]
    for any $\chi\in \mathbb{Z}_p$ and hence
    \[\sum_{\chi\in B_t}\|\chi (x_1+\dots+x_k)\|_p^2\le k\sum_{\chi\in B_t}\|\chi x_1\|_p^2+\dots+k\sum_{\chi\in B_t}\|\chi x_k\|_p^2\le k\cdot k\cdot \delta|B_t|=k^2 \delta|B_t|.\]
    This shows that $x_1+\dots+x_k\in Q_{t, k^2 \delta}$ for all $x_1,\dots,x_k\in Q_{t,\delta}$.
\end{proof}

We are now ready to prove Lemma \ref{lem:tildeC}.
\begin{proof}[Proof of Lemma \ref{lem:tildeC}]
We may assume $|B_t|\ge 2$, since otherwise the desired inequality trivially holds. Note that then Lemma \ref{lem:dual-est} in particular implies $|Q_{t, 1/200}|\le (5/4)\cdot (p/2)=(5/8)\cdot p<p$.%

Now, let $k=\lfloor \sqrt{m/(2000t)}\rfloor$. Recalling our assumption that $t\le m/2000$, we can see that $\sqrt{m/(2000t)}\ge 1$ and $k=\lfloor \sqrt{m/(2000t)}\rfloor\ge \frac{1}{2}\sqrt{m/(2000t)}\ge 10^{-2}\sqrt{m/t}$, and furthermore also $k=\lfloor \sqrt{m/(2000t)}\rfloor\ge 1$. We have $k^2(10t/m)\le 1/200$, and hence by Lemma \ref{lem:sumset}, the $k$-fold sumset $kQ_{t,10t/m}$ of $Q_{t,10t/m}$ satisfies
\[kQ_{t,10t/m}\su Q_{t,k^2(10t/m)}\su Q_{t,1/200}.\]
In particular, we obtain $|kQ_{t,10t/m}|\le |Q_{t,1/200}|<p$, and hence $kQ_{t,10t/m}\ne \mathbb{Z}_p$.

Now, combining Fact \ref{fact:cauchy-davenport} and Lemma \ref{lem:lower-bound-Q} yields
\[|kQ_{t,10t/m}|\ge 1+k\cdot (|Q_{t,10t/m}|-1)\ge 1+k\cdot \Big(\frac{9}{10}|S|-1\Big)\ge k\cdot\frac{4}{5}|S|=\frac{4}{5}\cdot k\cdot |S|\]
(recalling that $|S|\ge 10^{7}$). Thus, we can conclude
\[\frac{4}{5}\cdot k\cdot |S|\le |kQ_{t,10t/m}|\le |Q_{t,1/200}|\le  \frac{5}{4}\cdot \frac{p}{|B_t|},\]
where in the last step we used Lemma \ref{lem:dual-est}. Thus, we obtain
\[|B_t|\le \frac{25}{16}\cdot \frac{p}{k\cdot |S|}\le 2\cdot \frac{p}{10^{-2}\sqrt{m/t} \cdot |S|} =\frac{200 p\sqrt{t}}{|S|\sqrt{m}}\le 1+\frac{200 p\sqrt{t}}{|S|\sqrt{m}}.\qedhere\]
\end{proof}

Finally, we give the proof of Theorem \ref{thm:anticonc}.
\begin{proof}[Proof of Theorem \ref{thm:anticonc}]
Let $z\in \mathbb{Z}_p$. Recall from (\ref{eq:prob}) that
\[\mathbb{P}[\Sigma(R)=z] \le \frac{1}{p} \mathbb{E}_\mathcal{S}[|A_0|]+\frac{1}{p}\sum_{\ell= 0}^{\lfloor \log_2 m\rfloor} \mathbb{E}_\mathcal{S}[|A_{2^{\ell}}|] \cdot \exp(-2^{\ell}).\]
For every positive integer $t$, by Lemmas \ref{lem:est1} and \ref{lem:est2}, we have
\[\mathbb{E}_\mathcal{S}\Big[|\{\chi\in \mathbb{Z}_p\setminus\{0\}: \psi(\chi)<2t\}|\Big]\le \frac{|(\mathbb{Z}_p\setminus\{0\})\setminus D_t|}{|S|^9}+\frac{|D_t\setminus B_{2000t}|}{|S|^9}+|B_{2000t}\setminus\{0\}|\le \frac{p}{|S|^9}+|B_{2000t}\setminus\{0\}|.\]
Since $0\in B_{2000t}$ for all positive integers $t$, we have $|B_{2000t}\setminus\{0\}|=|B_{2000t}|-1$, and so for any positive integer $t\le m/2000^2$ we can conclude
\[\mathbb{E}_\mathcal{S}\Big[|\{\chi\in \mathbb{Z}_p\setminus\{0\}: \psi(\chi)<2t\}|\Big]\le  \frac{p}{|S|^9}+|B_{2000t}|-1\le \frac{p}{|S|^9}+\frac{200p\sqrt{2000t}}{|S|\sqrt{m}}\le \frac{10^4p\sqrt{t}}{|S|\sqrt{m}},\]
where the penultimate step is by Lemma \ref{lem:tildeC} (and in the last step we used $\sqrt{m}\le m\le |S|$).

Recall that for $\chi=0\in \mathbb{Z}_p$ we always have $\psi(\chi)=0$, and hence $\chi\in A_0$ (and $\chi\not\in A_t$ for all positive integers $t$). Thus, we can conclude (recalling that $m\ge 10^7$)
\[\mathbb{E}_\mathcal{S}[|A_0|]\le 1+\mathbb{E}_\mathcal{S}\Big[|\{\chi\in \mathbb{Z}_p\setminus\{0\}: \psi(\chi)<2\}|\Big]\le 1+\frac{10^4p}{|S|\sqrt{m}}\]
and
\[\mathbb{E}_\mathcal{S}[|A_t|]=\mathbb{E}_\mathcal{S}\Big[|\{\chi\in \mathbb{Z}_p\setminus\{0\}: t\le \psi(\chi)<2t\}|\Big]\le \mathbb{E}_\mathcal{S}\Big[|\{\chi\in \mathbb{Z}_p\setminus\{0\}: \psi(\chi)<2t\}|\Big]\le \frac{10^4p\sqrt{t}}{|S|\sqrt{m}}\]
for any positive integer $t\le m/2^{22}\le m/2000^2$. Finally, note that for any positive integer $t$, we also have the trivial bound $\mathbb{E}_\mathcal{S}[|A_t|]\le p$.

Plugging these bounds in, we obtain
\begin{align*}
\mathbb{P}[\Sigma(R)=z] &\le \frac{1}{p} \mathbb{E}_\mathcal{S}[|A_0|]+\frac{1}{p}\sum_{\ell= 0}^{\lfloor \log_2 m\rfloor} \mathbb{E}_\mathcal{S}[|A_{2^{\ell}}|] \cdot \exp(-2^{\ell})\\
&\le \frac{1}{p}\Bigg(1+\frac{10^4p}{|S|\sqrt{m}}\Bigg)+\frac{1}{p}\sum_{\ell= 0}^{\lfloor \log_2 m\rfloor-22}\frac{10^4p2^{\ell/2}}{|S|\sqrt{m}}\cdot \exp(-2^{\ell})+\frac{1}{p}\sum_{\ell= \lfloor \log_2 m\rfloor-21}^{\lfloor \log_2 m\rfloor}p\cdot \exp(-2^{\ell})\\
&\le \frac{1}{p}+\frac{10^4}{|S|\sqrt{m}}+\frac{10^4}{|S|\sqrt{m}}\sum_{\ell= 0}^{\infty}\exp(\ell/2-2^{\ell})+22\cdot \exp(-m/2^{22})\\
&\le \frac{1}{p}+\frac{10^4}{|S|\sqrt{m}}+\frac{10^4}{|S|\sqrt{m}}\cdot 2+\frac{22}{|S|^4}\le \frac{1}{p}+\frac{30022}{|S|\sqrt{m}}\le \frac{1}{p}+\frac{C}{|S|\sqrt{m}}.
\end{align*}
Here, for the inequality at the start of the last line we used $\sum_{\ell= 0}^{\infty}\exp(\ell/2-2^{\ell})\le \sum_{\ell= 0}^{\infty}\exp(\ell/2-2\ell)\le \sum_{\ell= 0}^{\infty}\exp(-\ell)\le \sum_{\ell= 0}^{\infty}2^{-\ell}=2$, as well as our assumption $m\ge C\log |S|= 2^{24}\log |S|$. For the next inequality, we used $\sqrt{m}\le m\le |S|$.
\end{proof}

\section{Combinatorial anticoncentration deductions} \label{sec:anticonc-de}

In this section, we use combinatorial arguments to deduce Corollary \ref{coro:anticonc} from Theorem \ref{thm:anticonc}, and to also derive another corollary that we will need in our proof of Theorem \ref{thm:main}.

We start with the following simple observation.

\begin{lem}\label{lem:trivial-anticoncentration-bound}
Let $p$ be a prime, let $S\subseteq \mathbb{Z}_p$, and let $m\le |S|$ be a positive integer. Now, let $R\su S$ be a uniformly random subset of $S$ of size $m$. Then we have
    \[
        \max_z \mathbb{P}[\Sigma(R)=z] \le \frac{1}{|S|-m+1},
    \]
\end{lem}

\begin{proof}
    We sample the uniformly random subset $R\su S$ of size $R$ by first sampling a uniformly random subset $R'\su S$ of size $m-1$, and then sampling a uniformly random element $r\in S\setminus R'$, setting $R=R'\cup\{r\}$.

     Conditional on any outcome of $R'$, for any $z\in \mathbb{Z}_p$, we have
    \[
        \mathbb{P}[\Sigma(R)=z\mid R'] =\mathbb{P}[r=z-\Sigma(R')\mid R'] \le \frac{1}{|S\setminus R'|}=\frac{1}{|S|-m+1}.
    \]
Thus, for any $z\in \mathbb{Z}_p$, we can conclude
    \[
        \mathbb{P}[\Sigma(R)=z] \le  \frac{1}{|S|-m+1}.\qedhere
    \]
\end{proof}

Now, we are ready to deduce Corollary \ref{coro:anticonc} from Lemma \ref{lem:trivial-anticoncentration-bound} and Theorem \ref{thm:anticonc}.

\begin{proof}[Proof of Corollary \ref{coro:anticonc}]
Let the constant $C>0$ be as in Theorem \ref{thm:main}. We may assume that $|S|$ is sufficiently large with respect to $\eps$, since the statement in Corollary \ref{coro:anticonc} trivially holds for small $|S|$ (by choosing a large constant $C'_\eps>0$). In particular, we may assume $|S|\ge (4000C/\eps)\cdot (\log |S|)^2$ (and $|S|\ge 10$).

If $m\le C\log |S|$, the assertion of Corollary \ref{coro:anticonc} follows directly from Lemma \ref{lem:trivial-anticoncentration-bound}, since then (noting that $m\le C\log |S|\le |S|/2$)
\[\max_z \mathbb{P}[\Sigma(R)=z] \le \frac{1}{|S|-m+1}\le \frac{1}{|S|/2}\le \frac{2\sqrt{C\log |S|}}{|S|\sqrt{m}}\le \frac{1}{p} +\frac{2\sqrt{C}\cdot \sqrt{\log |S|}}{|S|\sqrt{m}}.\]
Furthermore, if $C\log |S|\le m\le 10^{-3}|S|/\log |S|$, the assertion of Corollary \ref{coro:anticonc} follows directly from Theorem \ref{thm:anticonc}. We may therefore assume that $m\ge 10^{-3}|S|/\log |S|$.

Let $m_2=\lfloor \eps 10^{-3}|S|/\log |S|\rfloor\ge (\eps/2) 10^{-3}|S|/\log |S|$, and note that $m_2\le m$. Letting $m_1=m-m_2$, we may now sample the uniformly random subset $R\su S$ of size $m$ by first sampling a uniformly random subset $R_1\su S$ of size $m_1$, and then sampling a uniformly random subset $R_2\su S\setminus R_1$ of size $m_2$, setting $R=R_1\cup R_2$ (then $R$ has size $m_1+m_2=m$).

Note for any outcome of $R_1$, we have $|S\setminus R_1|\ge |S|-m\ge |S|-(1-\eps)|S|=\eps|S|$ and hence $m_2\le \eps 10^{-3}|S|/\log |S|\le 10^{-3}|S\setminus R_1|/\log |S\setminus R_1|$, as well as $m_2\ge (\eps/2) 10^{-3}|S|/\log |S|\ge C\log |S|\ge C\log |S\setminus R_1|$. Therefore, conditional on any outcome of $R_1$, by Theorem \ref{thm:anticonc}, for any $z\in \mathbb{Z}_p$, we have
    \[
        \mathbb{P}[\Sigma(R)=z\mid R_1] =\mathbb{P}[\Sigma(R_2)=z-\Sigma(R_1)\mid R_1] \le \frac{1}{p} + \frac{C}{|S\setminus R_1|\sqrt{m_2}}.
    \]
Noting that $|S\setminus R_1|\sqrt{m_2}\ge \eps |S|\cdot \sqrt{(\eps/2) 10^{-3}|S|/\log |S|}\ge \eps^{3/2}|S|^{3/2}/(50\sqrt{\log |S|})$, we can conclude that
    \[
        \mathbb{P}[\Sigma(R)=z\mid R_1] \le \frac{1}{p} + \frac{50C\eps^{-3/2}\sqrt{\log |S|}}{|S|^{3/2}}\le \frac{1}{p} +\frac{50C\eps^{-3/2}\sqrt{\log |S|}}{|S|\sqrt{m}}
    \]
for any outcome of $R_1$ and any $z\in \mathbb{Z}_p$. This shows that
    \[
        \mathbb{P}[\Sigma(R)=z] \le  \frac{1}{p} +\frac{50C\eps^{-3/2}\sqrt{\log |S|}}{|S|\sqrt{m}}
    \]
for any $z\in \mathbb{Z}_p$, finishing the proof.
\end{proof}

In the proof of our main result in Theorem \ref{thm:main}, we will use the following corollary of Corollary \ref{coro:anticonc} several times. Rather than with a single random subset $R\su S$ of a given size, this corollary is concerned with a uniformly random chain of subsets $R_1\su R_1\su \dots\su R_k\su S$ of given sizes.

\begin{cor}\label{coro:anticonc-chain}
    For every positive integer $k$, there exists a constant $C_k>0$ such that the following holds. Let $p$ be a prime, let $S\subseteq \mathbb{Z}_p$ be a subset of size $|S|\ge 2$, and consider integers $1\le m_1<\dots<m_k< |S|$. Now, let $R_1\su \dots\su R_k\su S$ be a uniformly random chain of subsets of $S$ satisfying $|R_i|=m_i$ for $i=1,\dots,k$. Then for any $z_1,\dots,z_k\in \mathbb{Z}_p$, we have
    \[
         \mathbb{P}[\Sigma(R_i)=z_i \text{ for }i=1,\dots,k] \le \sum_{j=0}^k \, \prod_{i\in \{0,\dots,k\}\setminus\{j\}}\Bigg(\frac{1}{p} + \frac{C_k\sqrt{\log |S|}}{|S|\sqrt{m_{i+1}-m_i}}\Bigg),
    \]
    where we set $m_0=0$ and $m_{k+1}=|S|$.
\end{cor}
\begin{proof}
     Let $\eps=1/(k+1)$, and let us define $C_k=C'_\eps/\eps$ for the constant $C'_\eps$ in Corollary \ref{coro:anticonc}. Also let $z_1,\dots,z_k\in \mathbb{Z}_p$.
    
    As $(m_1-m_0)+(m_2-m_1)+\dots+(m_{k+1}-m_k)=m_{k+1}-m_0=|S|$, we can fix an index $j\in \{0,\dots,k\}$ with $m_{j+1}-m_j\ge |S|/(k+1)$. It now suffices to prove
    \[
         \mathbb{P}[\Sigma(R_i)=z_i \text{ for }i=1,\dots,k] \le \prod_{i\in \{0,\dots,k\}\setminus\{j\}} \,\,\Bigg(\frac{1}{p} + \frac{(C'_\eps/\eps)\sqrt{\log |S|}}{|S|\sqrt{m_{i+1}-m_i}}\Bigg).
    \]

    For $i=j+1,\dots,k$, let $R'_i=S\setminus R_i$, and note that then we have $R'_k\su R'_{k-1}\su \dots\su R'_{j+1}\su S\setminus R_j$ and $|R'_i|=|S|-m_i$ for $i=j+1,\dots,k$. Furthermore, conditional on any outcome of $R_1,\dots,R_j$, the chain of sets $R'_k\su R'_{k-1}\su \dots\su R'_{j+1}\su S\setminus R_j$ is uniformly random among all such nested chains in $S\setminus R_j$ with the appropriate sizes $|R'_i|=|S|-m_i$ for $i=j+1,\dots,k$.
    
    We expose the random chain $R_1\su R_1\su \dots\su R_k\su S$ by first exposing $R_1$ (this is a uniformly random subset of $S$ of size $m_1$), then $R_2\setminus R_1$ (this is now  a uniformly random subset of $S\setminus R_1$ of size $m_2-m_1$), and so on until we expose $R_j\setminus R_{j-1}$ (which is a uniformly random subset of $S\setminus R_{j-1}$ of size $m_j-m_{j-1}$). At this point we have determined $R_1,\dots,R_j$. Now, we expose $R'_k$ (this is a uniformly random subset of $S\setminus R_j$ of size $|S|-m_k$), then $R'_{k-1}\setminus R'_k$ (this is a uniformly random subset of $S\setminus (R_j\cup R'_k)$ of size $(|S|-m_{k-1})-(|S|-m_k)=m_k-m_{k-1}$), and so on until we have exposed the entire chain $R'_k\su R'_{k-1}\su \dots\su R'_{j+1}$ as well. This now determines $R_1\su R_1\su \dots\su R_k$.

    Observe that we have $\Sigma(R_i)=z_i$ for $i=1,\dots,k$ if and only if all of the following conditions hold: $\Sigma(R_1)=z_1$, and $\Sigma(R_i\setminus R_{i-1})=z_i-z_{i-1}$ for $i=2,\dots,j$, and $\Sigma(R'_k)=\Sigma(S)-z_{k}$, and $\Sigma(R'_{i}\setminus R'_{i+1})=z_{i+1}-z_i$ for $i=j+1,\dots,k-1$. This means that at every step of the exposure procedure above, we choose a uniformly random subset of the not-yet-chosen elements of a certain prescribed size, and we need to bound the probability of attaining a certain sum (conditional on our previously exposed information). At each such step, we can use Corollary \ref{coro:anticonc} to bound this probability. For this, we crucially use that at any step we choose at most a $(1-\eps)$-fraction of the not-yet-chosen elements. Indeed, this holds since at the very end of the procedure $m_{j+1}-m_j\ge |S|/(k+1)=\eps |S|$ elements remain unchosen (namely, the elements in $R_{j+1}\setminus R_j$), and so we cannot reduce the number of leftover elements by a factor smaller than $\eps$ at any step. Also note that at any step the set of not-yet-chosen elements has size at least $\eps |S|$.

    Thus, at any step, conditional on the previously exposed information, the probability of our random subset attaining the desired sum can indeed be bounded by Corollary \ref{coro:anticonc}. More specifically, the probability of having $\Sigma(R_1)=z_1$ is at most
    \[\frac{1}{p} + \frac{C'_\eps\sqrt{\log |S|}}{\eps|S|\sqrt{m_1}}=\frac{1}{p} + \frac{C'_\eps\sqrt{\log |S|}}{\eps|S|\sqrt{m_{1}-m_0}}.\]
    For any $i=2,\dots,j$, the probability of having $\Sigma(R_i\setminus R_{i-1})=z_i-z_{i-1}$, conditional on the previously exposed information, is at most
    \[\frac{1}{p} + \frac{C'_\eps\sqrt{\log |S|}}{\eps|S|\sqrt{m_{i}-m_{i-1}}}.\]
    The probability of having $\Sigma(R'_k)=\Sigma(S)-z_{k}$, conditional on the previously exposed information determining $S\setminus R_j$, is at most
    \[\frac{1}{p} + \frac{C'_\eps\sqrt{\log |S|}}{\eps|S|\sqrt{|S|-m_{k}}}=\frac{1}{p} + \frac{C'_\eps\sqrt{\log |S|}}{\eps|S|\sqrt{m_{k+1}-m_{k}}}.\]
    And, finally, for $i=j+1,\dots,k-1$, the probability of having $\Sigma(R'_{i}\setminus R'_{i+1})=z_{i+1}-z_i$, conditional on the previously exposed information, is at most
    \[\frac{1}{p} + \frac{C'_\eps\sqrt{\log |S|}}{\eps|S|\sqrt{m_{i+1}-m_{i}}}.\]
    Altogether this shows that
    \begin{align*}
         &\mathbb{P}[\Sigma(R_i)=z_i \text{ for }i=1,\dots,k]\\
         &\qquad\le \Bigg(\frac{1}{p} + \frac{C'_\eps\sqrt{\log |S|}}{\eps|S|\sqrt{m_{1}-m_0}}\Bigg)\cdot \prod_{i=2}^{j}\Bigg(\frac{1}{p} + \frac{C'_\eps\sqrt{\log |S|}}{\eps|S|\sqrt{m_{i}-m_{i-1}}}\Bigg)\cdot \Bigg(\frac{1}{p} + \frac{C'_\eps\sqrt{\log |S|}}{\eps|S|\sqrt{m_{k+1}-m_k}}\Bigg) \\
         &\qquad\qquad\qquad\qquad\qquad\qquad\qquad\qquad\qquad\qquad\qquad\qquad\qquad\qquad\cdot\prod_{i=j+1}^{k}\Bigg(\frac{1}{p} + \frac{C'_\eps\sqrt{\log |S|}}{\eps|S|\sqrt{m_{i+1}-m_{i}}}\Bigg)\\
         &\qquad=\prod_{i\in \{0,\dots,k\}\setminus\{j\}} \,\,\Bigg(\frac{1}{p} + \frac{(C'_\eps/\eps)\sqrt{\log |S|}}{|S|\sqrt{m_{i+1}-m_i}}\Bigg).\qedhere
    \end{align*}
\end{proof}

We end this section with the following bound for the sum of the terms on the right-hand side of the inequality in Corollary \ref{coro:anticonc-chain} obtained for different choices of $m_1<\dots<m_k$.

\begin{lem}\label{lem:sum-of-weird-function}
    Let $k$ be a positive integer, and let $C_k>0$ be the constant in Corollary \ref{coro:anticonc-chain}. Furthermore, let $p$ be a prime, and let $S\subseteq \mathbb{Z}_p$ be a subset of size $|S|\ge 2$. As in Corollary \ref{coro:anticonc-chain}, let $m_0=0$ and $m_{k+1}=|S|$. Then, we have
    \[\sum_{1\le m_1<\dots<m_k<|S|}\,\sum_{j=0}^k \,\prod_{i\in \{0,\dots,k\}\setminus\{j\}}\Bigg(\frac{1}{p} + \frac{C_k\sqrt{\log |S|}}{|S|\sqrt{m_{i+1}-m_i}}\Bigg)\le (k+1)\cdot \Bigg(\frac{|S|}{p} + \frac{2C_k\sqrt{\log |S|}}{|S|^{1/2}}\Bigg)^k,\]
    where the sum on the left-hand side is taken over all $k$-tuples $(m_1,\dots,m_k)\in \{1,\dots,|S|\}^k$ satisfying $1\le m_1<\dots<m_k<|S|$.
\end{lem}
\begin{proof}
    We first show by induction that for $h=0,\dots,k$ we have
    \begin{equation}\label{eq:inequality-simple-induct}
        \sum_{1\le m_1<\dots<m_h<|S|}\,\prod_{i=0}^{h-1} \,\,\Bigg(\frac{1}{p} + \frac{C_k\sqrt{\log |S|}}{|S|\sqrt{m_{i+1}-m_i}}\Bigg)\le \Bigg(\frac{|S|}{p} + \frac{2C_k\sqrt{\log |S|}}{|S|^{1/2}}\Bigg)^h.
    \end{equation}
    For $h=0$, the sum on the right-hand side of \eqref{eq:inequality-simple-induct} is has exactly one term, corresponding to the unique empty $0$-tuple, and this term is $1$ (as an empty product). Thus, for $h=0$ the inequality reads $1\le 1$, which is true. So let us now assume that we already showed \eqref{eq:inequality-simple-induct} for some $h\in \{0,\dots,k-1\}$. Now, for $h+1$, we have
    \begin{align*}
        &\sum_{1\le m_1<\dots<m_{h+1}<|S|}\,\prod_{i=0}^{h} \,\,\Bigg(\frac{1}{p} + \frac{C_k\sqrt{\log |S|}}{|S|\sqrt{m_{i+1}-m_i}}\Bigg)\\
        &\qquad\qquad=\sum_{1\le m_1<\dots<m_{h}<|S|}\Bigg(\prod_{i=0}^{h-1} \,\,\Bigg(\frac{1}{p} + \frac{C_k\sqrt{\log |S|}}{|S|\sqrt{m_{i+1}-m_i}}\Bigg)\cdot \sum_{m_{h+1}=m_{h}+1}^{|S|-1}\Bigg(\frac{1}{p} + \frac{C_k\sqrt{\log |S|}}{|S|\sqrt{m_{h+1}-m_h}}\Bigg)\Bigg)\\
        &\qquad\qquad\le \sum_{1\le m_1<\dots<m_{h}<|S|}\Bigg(\prod_{i=0}^{h-1} \,\,\Bigg(\frac{1}{p} + \frac{C_k\sqrt{\log |S|}}{|S|\sqrt{m_{i+1}-m_i}}\Bigg)\cdot \Bigg(\frac{|S|}{p} + \frac{2C_k\sqrt{\log |S|}}{|S|^{1/2}}\Bigg)\Bigg)\\
        &\qquad\qquad= \Bigg(\frac{|S|}{p} + \frac{2C_k\sqrt{\log |S|}}{|S|^{1/2}}\Bigg)\cdot \sum_{1\le m_1<\dots<m_{h}<|S|}\prod_{i=0}^{h-1} \,\,\Bigg(\frac{1}{p} + \frac{C_k\sqrt{\log |S|}}{|S|\sqrt{m_{i+1}-m_i}}\Bigg)\\
        &\qquad\qquad\le  \Bigg(\frac{|S|}{p} + \frac{2C_k\sqrt{\log |S|}}{|S|^{1/2}}\Bigg)\cdot \Bigg(\frac{|S|}{p} + \frac{2C_k\sqrt{\log |S|}}{|S|^{1/2}}\Bigg)^h=\Bigg(\frac{|S|}{p} + \frac{2C_k\sqrt{\log |S|}}{|S|^{1/2}}\Bigg)^{h+1},
    \end{align*}
    where the second inequality holds by the inductive assumption, and the first inequality follows from
    \[\sum_{m_{h+1}=m_{h}+1}^{|S|-1}\frac{C_k\sqrt{\log |S|}}{|S|\sqrt{m_{h+1}-m_h}}\le \sum_{t=1}^{|S|}\frac{C_k\sqrt{\log |S|}}{|S|\sqrt{t}}\le \frac{C_k\sqrt{\log |S|}}{|S|}\cdot 2\sqrt{|S|}=\frac{2C_k\sqrt{\log |S|}}{|S|^{1/2}}.\]
    This proves \eqref{eq:inequality-simple-induct}, as claimed.
    
    Let us now turn to the inequality in the lemma statement. To show this inequality, it clearly suffices to show
    \[\sum_{1\le m_1<\dots<m_k<|S|} \,\prod_{i\in \{0,\dots,k\}\setminus\{j\}} \,\,\Bigg(\frac{1}{p} + \frac{C_k\sqrt{\log |S|}}{|S|\sqrt{m_{i+1}-m_i}}\Bigg)\le \Bigg(\frac{|S|}{p} + \frac{2C_k\sqrt{\log |S|}}{|S|^{1/2}}\Bigg)^k\]
    for all $j=0,\dots,k$. So let us fix $j\in \{0,\dots,k\}$. Now, we have
    \begin{align*}
        &\sum_{1\le m_1<\dots<m_k<|S|} \,\prod_{i\in \{0,\dots,k\}\setminus\{j\}} \Bigg(\frac{1}{p} + \frac{C_k\sqrt{\log |S|}}{|S|\sqrt{m_{i+1}-m_i}}\Bigg)\\
        &\qquad\le \sum_{1\le m_1<\dots<m_j<|S|}\,\,\sum_{1\le m_{j+1}<\dots<m_k<|S|} \,\prod_{i=0}^{j-1} \Bigg(\frac{1}{p} + \frac{C_k\sqrt{\log |S|}}{|S|\sqrt{m_{i+1}-m_i}}\Bigg)\cdot \prod_{i=j+1}^{k} \Bigg(\frac{1}{p} + \frac{C_k\sqrt{\log |S|}}{|S|\sqrt{m_{i+1}-m_i}}\Bigg)\\
        &\qquad= \sum_{1\le m_1<\dots<m_j<|S|}\,\prod_{i=0}^{j-1} \Bigg(\frac{1}{p} + \frac{C_k\sqrt{\log |S|}}{|S|\sqrt{m_{i+1}-m_i}}\Bigg)\cdot \sum_{1\le m_{j+1}<\dots<m_k<|S|} \, \prod_{i=j+1}^{k} \Bigg(\frac{1}{p} + \frac{C_k\sqrt{\log |S|}}{|S|\sqrt{m_{i+1}-m_i}}\Bigg)\\
        &\qquad= \sum_{1\le m_1<\dots<m_j<|S|}\,\prod_{i=0}^{j-1} \Bigg(\frac{1}{p} + \frac{C_k\sqrt{\log |S|}}{|S|\sqrt{m_{i+1}-m_i}}\Bigg)\cdot \sum_{1\le m'_{1}<\dots<m'_{k-j}<|S|} \prod_{i=0}^{k-j-1} \Bigg(\frac{1}{p} + \frac{C_k\sqrt{\log |S|}}{|S|\sqrt{m'_{i+1}-m'_i}}\Bigg)\\
        &\qquad\le \Bigg(\frac{|S|}{p} + \frac{2C_k\sqrt{\log |S|}}{|S|^{1/2}}\Bigg)^j\cdot \Bigg(\frac{|S|}{p} + \frac{2C_k\sqrt{\log |S|}}{|S|^{1/2}}\Bigg)^{k-j}=\Bigg(\frac{|S|}{p} + \frac{2C_k\sqrt{\log |S|}}{|S|^{1/2}}\Bigg)^k,
    \end{align*}
    defining $m'_0=0$. Here, in the third step we substituted $m'_1=|S|-m_{k}$ and $m'_2=|S|-m_{k-1}$, and so on until $m'_{k-j}=|S|-m_{j+1}$. Furthermore, in the penultimate step, we used \eqref{eq:inequality-simple-induct}.
\end{proof}

\section{Rearrangement conjecture}\label{sec:rearrangement}

In this section, we prove Theorem \ref{thm:main}, using Corollary \ref{coro:anticonc}. Let $0<\alpha<1/2$ (note that in Theorem~\ref{thm:main} we may assume without loss of generality that $\alpha<1/2$, since the statement for some value of $\alpha$ automatically implies the same statement for all larger values of $\alpha$). Let $D=\lceil 3/\alpha\rceil$ and let $C_\alpha$ be large enough such that $C_\alpha\ge (10^4 \cdot 2^{40D})^{1/\alpha}$ and $C_\alpha\ge (D+1)\cdot 2^D\cdot D^{14D^2}\ge 100\cdot (5D)^{2D}\ge (40D)^D$ and such that we have $4\max(C_D,C_1)\sqrt{\log n}/n^{1/2}\le n^{-\alpha}$ for all $n\ge C_\alpha$ (where $C_1$ and $C_D$ denote the constants in Corollary \ref{coro:anticonc-chain} for $k=1$ and $k=D$).

Now, let $S\subseteq \mathbb{Z}_p\setminus\{0\}$ be as in Theorem \ref{thm:main}, meaning that $C_\alpha\le |S|\le p^{1-\alpha}$. Then in particular we have $|S|/p\le p^{-\alpha}\le |S|^{-\alpha}$ as well as $4C_1\sqrt{\log |S|}/|S|^{1/2}\le |S|^{-\alpha}$ and $4C_D\sqrt{\log |S|}/|S|^{1/2}\le |S|^{-\alpha}$.

We need to show that there is an ordering $s_1,s_2,\dots,s_{|S|}$ of the elements of $S$ such that all partial sums $s_1,s_1+s_2,\dots,s_1+s_2+\dots+s_{|S|}$ are distinct. In other words, this means that all the partial sums $s_a+s_{a+1}+\dots+s_b$ must be non-zero in $\mathbb{Z}_p$ for all $2\le a\le b\le |S|$.

An ordering $s_1,s_2,\dots,s_{|S|}$ of the elements of $S$ corresponds to a bijection $\sigma:\{1,\dots, |S|\}\to S$, where $\sigma(a)=s_a$ for all $a\in \{1,\dots, |S|\}$. For a bijection $\sigma:\{1,\dots, |S|\}\to S$,  we define
\[\Sigma(\sigma,[a,b])=\sum_{i=a}^b \sigma(i)=\sigma(a)+\sigma(a+1)+\dots+\sigma(b)\]
for all $1\le a< b\le |S|$. This is precisely $s_a+s_{a+1}+\dots+s_b$ for the corresponding ordering $s_1,s_2,\dots,s_{|S|}$ of the elements of $S$. Thus, to prove Theorem \ref{thm:main}, our task is to find a bijection $\sigma:\{1,\dots, |S|\}\to S$ such that $\Sigma(\sigma,[a,b])\ne 0$ for all $2\le a< b\le |S|$.

Our starting point is to consider a uniformly random bijection $\sigma:\{1,\dots, |S|\}\to S$. Typically, for such a  random bijection there will be some (but only few) pairs $(a,b)$ with $2\le a< b\le |S|$ and  $\Sigma(\sigma,[a,b])= 0$. Instead of $\sigma$ itself, we will therefore consider $\sigma \circ \pi$ for a suitable permutation $\pi: \{1,\dots, |S|\}\to\{1,\dots, |S|\}$ obtained by composing certain transpositions in order to change $\Sigma(\sigma,[a,b])$ for those intervals $[a,b]$ where we have $\Sigma(\sigma,[a,b])= 0$.

Given a bijection $\sigma:\{1,\dots, |S|\}\to S$, let $B(\sigma)\su \{3,\dots, |S|\}$ be the set consisting of the right endpoints of all intervals $[a,b]$ with $2\le a<b\le |S|$ and $\Sigma(\sigma,[a,b])= 0$. More formally, let $B(\sigma)\su \{1,\dots, |S|\}$ be the set of all $b\in \{1,\dots, |S|\}$ such that there exists $a\in \{2,\dots, |S|\}$ with $a<b$ and $\Sigma(\sigma,[a,b])= 0$ (note that then automatically $b\ge 3$).

We will now consider various bad events (namely the events $\mathcal{B}_1$,  $\mathcal{B}_2$ and $\mathcal{B}_3$ defined below), and show via our anticoncentration results in Corollary \ref{coro:anticonc-chain} that for a uniformly random bijection $\sigma:\{1,\dots, |S|\}\to S$ these bad events occur with very small probability. In order to prove  Theorem \ref{thm:main}, we can then fix a bijection $\sigma:\{1,\dots, |S|\}\to S$ such that none of these events holds, and take this as a starting point to construct the desired ordering $s_1,\dots,s_{|S|}$ of the elements in $S$.

\begin{lem}\label{lem:property1-new}
    For a uniformly random bijection $\sigma:\{1,\dots, |S|\}\to S$, let $\mathcal{B}_1$ be the event that there is some $b\in B(\sigma)$ with $|S|-30D\le b\le |S|$. Then we have $\mathbb{P}[\mathcal{B}_1]\le 1/100$.
\end{lem}

\begin{lem}\label{lem:property3-new}
    For a uniformly random bijection $\sigma:\{1,\dots, |S|\}\to S$, let $\mathcal{B}_2$ be the event that there exists $z\in\{1,\dots,|S|\}$ with $|B(\sigma)\cap \{z-10D,\dots,z+10D\}|> D$. Then we have $\mathbb{P}[\mathcal{B}_2]\le 3/100$.
\end{lem}

We will prove Lemmas \ref{lem:property1-new} and \ref{lem:property3-new} further below (towards the end of this section). The proofs of these lemmas rely on our anticoncentration results established in the previous sections.

For a permutation $\pi:\{1,\dots, |S|\}\to\{1,\dots, |S|\}$, let $\operatorname{Fix}(\pi)=\{t\in \{1,\dots, |S|\} : \pi(t)=t\}$ denote the set of fixed points of $\pi$. For any $x,y\in \{1,\dots, |S|\}$ with $x<y$, let $\pi_{x,y}:\{1,\dots, |S|\}\to\{1,\dots, |S|\}$ denote the transposition interchanging $x$ and $y$ (i.e.\ $\pi_{x,y}(x)=y$, $\pi_{x,y}(y)=x$ and $\pi_{x,y}(t)=t$ for all $t\in \{1,\dots, |S|\}\setminus \{x,y\}$).

Let us say that a collection $P$ of pairs $(x,y)$ with $x,y\in \{1,\dots, |S|\}$ and $x<y$ is \emph{admissible} if all the elements of $\{1,\dots, |S|\}$ appearing in the pairs in $P$ are distinct, and we have $y-x\le 5D$ for all $(x,y)\in P$.  To any such admissible $P$, we can associate the bijection $\pi_P:\{1,\dots, |S|\}\to\{1,\dots, |S|\}$ obtained by composing the transpositions $\pi_{x,y}$ for all $(x,y)\in P$ (note that here the order of these transpositions does not matter, since they all commute with each other as their supports are disjoint). We say that a permutation $\pi:\{1,\dots, |S|\}\to\{1,\dots, |S|\}$ is \emph{admissible}, if $\pi=\pi_P$ for some admissible  $P$. Also note that $P$ can be uniquely reconstructed from $\pi_P$, so the admissible permutations are in one-to-one correspondence with the admissible collections $P$ of pairs.

Intuitively, the admissible permutations are precisely those permutations that can be obtained by swapping a collection of disjoint pairs $(x,y)$ with $x<y\le x+5D$. For us, such permutations are relevant, because we will alter the bijection $\sigma:\{1,\dots, |S|\}\to S$ we are starting with in our argument via such a permutation $\pi$ (see also the discussion above), replacing $\sigma$ by $\sigma\circ \pi$. More precisely, $\pi$ will be obtained by composing carefully chosen transposition $\pi_{b,y}$ (with $b<y\le b+5D$) for each of the elements $b\in B(\sigma)$. We will chose these transpositions $\pi_{b,y}$ one by one for each $b\in B(\sigma)$.

For a bijection $\sigma:\{1,\dots, |S|\}\to S$, an element $b\in B(\sigma)$, and an admissible permutation $\pi:\{1,\dots, |S|\}\to\{1,\dots, |S|\}$, let us say that an element $y\in \{1,\dots, |S|\}$ with $b<y\le b+5D$ is \emph{blocked} for $\sigma$, $b$ and $\pi$, if we have $\Sigma(\sigma\circ \pi\circ \pi_{b,y}, [s,t])=0$ for some interval $[s,t]$ with $s,t\in \{2,\dots, |S|\}$ and $s<t$, such that $s\in \{b+1,\dots,y\}$ or $t\in \{b,\dots,y-1\}$.

Heuristically, if $y$ is blocked for $\sigma$, $b$ and $\pi$, then this means that we cannot replace $\sigma\circ \pi$ by $\sigma\circ \pi\circ \pi_{b,y}$ when trying to choose the desired transposition $\pi_{b,y}$ for the given element $b\in B(\sigma)$, because otherwise the interval $[s,t]$ would create a new problem.

\begin{lem}\label{lem:property4-new}
    For a uniformly random bijection $\sigma:\{1,\dots, |S|\}\to S$, let $\mathcal{B}_3$ be the event that there exists $b\in B(\sigma)$ and an admissible permutation $\pi:\{1,\dots, |S|\}\to\{1,\dots, |S|\}$ with $\{1,\dots,b-1\}\su \operatorname{Fix}(\pi)$ such that at least $2D$ different elements $y\in \{b+1,\dots, b+5D\}$ are blocked for $\sigma$, $b$ and $\pi$. Then we have $\mathbb{P}[\mathcal{B}_3]\le 1/25$.
\end{lem}

We will also  postpone the proof of Lemma \ref{lem:property4-new} to later (it again uses our anticoncentration results), in order to first see how Theorem \ref{thm:main} follows from Lemmas \ref{lem:property1-new} to \ref{lem:property4-new}.

\begin{proof}[Proof of Theorem \ref{thm:main}]
    For a uniformly random bijection $\sigma:\{1,\dots, |S|\}\to S$, the probability that one of the events $\mathcal{B}_1$, $\mathcal{B}_2$ and $\mathcal{B}_3$ in Lemmas \ref{lem:property1-new} to \ref{lem:property4-new} holds is at most  $\mathbb{P}[\mathcal{B}_1]+\mathbb{P}[\mathcal{B}_2]+\mathbb{P}[\mathcal{B}_3]\le 1/100+3/100+1/25=2/25<1$. In particular, there exists a bijection $\sigma:\{1,\dots, |S|\}\to S$ such that none of the events $\mathcal{B}_1$, $\mathcal{B}_2$ and $\mathcal{B}_3$ holds. Let us fix such a bijection $\sigma$.
    
    Now, let $B(\sigma)=\{b_1,\dots,b_\ell\}$ with $b_1>\dots> b_\ell$. Since the event $\mathcal{B}_1$ does not hold, we have $b_i<|S|-30D$ for $i=1,\dots,\ell$. Furthermore, since the event $\mathcal{B}_2$ does not hold, for every $z\in \{1,\dots, |S|\}$ there are at most $D$ different indices $i\in \{1,\dots,\ell\}$ with $|b_i-z|\le 10D$.

    We will now construct distinct $y_1,\dots,y_\ell\in \{1,\dots, |S|\}\setminus \{b_1,\dots,b_\ell\}$ such that $b_i<y_i\le b_i+5D$ for $i=1,\dots,\ell$ and such that the bijection $\sigma \circ \pi_{b_1,y_1}\circ \dots\circ\pi_{b_\ell,y_\ell}: \{1,\dots, |S|\}\to S$ corresponds to an ordering of the elements of $S$ with the desired conditions. In other words, this means that $\Sigma(\sigma \circ \pi_{b_1,y_1}\circ \dots\circ\pi_{b_\ell,y_\ell},[a,b])\ne 0$ for all $2\le a<b\le |S|$.

    We will construct $y_1,\dots,y_\ell$ one by one. At step $j$, we ensure that for all $2\le a<b\le |S|$ with $\Sigma(\sigma \circ \pi_{b_1,y_1}\circ \dots\circ\pi_{b_j,y_j},[a,b])= 0$ we have $b\in \{b_{j+1},\dots,b_\ell\}$. Furthermore, when choosing $y_j$, we also ensure that $y_j\in \{1,\dots, |S|\}\setminus \{b_1,\dots,b_\ell,y_1,\dots,y_{j-1}\}$ and $b_j<y_j\le b_j+5D$.

    Let us now assume that for some $j\in \{1,\dots,\ell\}$ we have already constructed distinct $y_1,\dots,y_{j-1}\in \{1,\dots, |S|\}\setminus \{b_1,\dots,b_\ell\}$ with $b_i<y_i\le b_i+5D$ for $i=1,\dots,j-1$, such that for all $2\le a<b\le |S|$ with $\Sigma(\sigma \circ \pi_{b_1,y_1}\circ \dots\circ\pi_{b_{j-1},y_{j-1}},[a,b])= 0$ we have $b\in \{b_{j},\dots,b_\ell\}$. Note that in the case $j=1$ the latter condition holds trivially, since for all $2\le a<b\le |S|$ with $\Sigma(\sigma,[a,b])= 0$ we have $b\in B(\sigma)=\{b_{1},\dots,b_\ell\}$.

    First, recall that $b_j<|S|-30D$, so we have $\{b_j+1,\dots,b_j+5D\}\su \{1,\dots, |S|\}$. Observing that the collection consisting of the pairs $(b_i,y_i)$ for $i=1,\dots,j-1$ is admissible, we can consider the admissible permutation $\pi^*=\pi_{b_1,y_1}\circ \dots\circ\pi_{b_{j-1},y_{j-1}}$. Because $b_1,\dots, b_{j-1}> b_j$, and consequently $y_1,\dots, y_{j-1}> b_j$, we have $\{1,\dots,b_j\}\su \operatorname{Fix}(\pi^*)$. Therefore, since $\sigma$ does not satisfy the event $\mathcal{B}_3$ in Lemma \ref{lem:property4-new}, at most $2D$ elements $y\in \{b_j+1,\dots,b_j+5D\}$ are blocked for $\sigma$, $b_j\in B(\sigma)$ and  $\pi^*$.

    We claim that there exists $y_j\in \{b_j+1,\dots,b_j+5D\}$ which is distinct from $b_1,\dots,b_\ell,y_1,\dots,y_{j-1}$ and not blocked for $\sigma$, $b_j$ and  $\pi^*$. Indeed, the interval $\{b_j+1,\dots,b_j+5D\}$ can contain $b_i$ for at most $D$ indices $i\in \{1,\dots,\ell\}$ (since we have $|b_i-b_j|\le 5D$ for all such $i$). Furthermore, for any index $i\in \{1,\dots,j-1\}$ with $y_i\in \{b_j+1,\dots,b_j+5D\}$ we have $|b_i-b_j|\le |b_i-y_i|+|y_i-b_j|\le 5D+5D=10D$, and so there can also be at most $D$ indices $i\in \{1,\dots,j-1\}$ with $y_i\in \{b_j,\dots,b_j+5D\}$.  Thus, the number of elements $y\in \{b_j+1,\dots,b_j+5D\}$, such that $y$ is distinct from $b_1,\dots,b_\ell,y_1,\dots,y_{j-1}$, and $y$ is not blocked for $\sigma$, $b_j$ and  $\pi^*$, is at least $5D-2D-D-D=D>0$. In particular, we can choose $y_j\in \{b_j+1,\dots,b_j+5D\}\su \{1,\dots, |S|\}$ distinct from $b_1,\dots,b_\ell,y_1,\dots,y_{j-1}$, such that $y_j$ is not blocked for $\sigma$, $b_j$ and  $\pi^*$.

    The latter condition means that there is no interval $[a,b]$ with $a,b\in \{2,\dots, |S|\}$ and $a<b$ and $\Sigma(\sigma\circ \pi^*\circ \pi_{b_j,y_j}, [a,b])=0$, such that $a\in \{b_j+1,\dots,y_j\}$ or $b\in \{b_j,\dots,y_j-1\}$. This means that for any $2\le a<b\le |S|$ with $\Sigma(\sigma \circ \pi^*\circ \pi_{b_j,y_j},[a,b])= 0$ we have $a\not\in\{b_j+1,\dots,y_j\}$ and $b\not\in \{b_j,\dots,y_j-1\}$. We claim that then either both of $b_j$ and $y_j$ or neither of $b_j$ and $y_j$ belong to the interval $[a,b]$. Indeed, if $b_j\in [a,b]$, then we have $a\le b_j< y_j$ as well as $b\ge b_j$ and therefore $b\ge y_j$ (as $b\not\in \{b_j,\dots,y_j-1\}$), implying that $y_j\in [a,b]$. If $y_j\in [a,b]$, then we have $b\ge y_j>b_j$ as well as $a\le y_j$ and therefore $a\le b_j$ (as $a\not\in\{b_j+1,\dots,y_j\}$), implying $b_j\in [a,b]$. Thus, for any $2\le a<b\le |S|$ with $\Sigma(\sigma \circ \pi^*\circ \pi_{b_j,y_j},[a,b])= 0$, either both of $b_j$ and $y_j$ or neither of $b_j$ and $y_j$ belong to the interval $[a,b]$, and so in the sum
    \[\Sigma(\sigma \circ \pi^*\circ \pi_{b_j,y_j},[a,b])= \sigma(\pi^*(\pi_{b_j,y_j}(a)))+\sigma(\pi^*(\pi_{b_j,y_j}(a+1)))+\dots+\sigma(\pi^*(\pi_{b_j,y_j}(b)))\]
    either both or neither of the terms $\sigma(\pi^*(\pi_{b_j,y_j}(b_j)))=\sigma(\pi^*(y_j))$ and $\sigma(\pi^*(\pi_{b_j,y_j}(y_j)))=\sigma(\pi^*(b_j))$ appear. This means that we have
    \[\Sigma(\sigma \circ \pi^*\circ \pi_{b_j,y_j},[a,b])= \sigma(\pi^*(a)))+\sigma(\pi^*(a+1))+\dots+\sigma(\pi^*(b))=\Sigma(\sigma \circ \pi^*,[a,b]).\]
    Hence we obtain
    \[\Sigma(\sigma \circ \pi_{b_1,y_1}\circ \dots\circ\pi_{b_{j-1},y_{j-1}},[a,b])=\Sigma(\sigma \circ \pi^*,[a,b])=\Sigma(\sigma \circ \pi^*\circ \pi_{b_j,y_j},[a,b])= 0,\]
    and so we must have $b\in \{b_{j},\dots,b_\ell\}$. As $b\not\in\{b_j,b_j+1,\dots,y_j-1\}$, we have $b\ne b_j$ and consequently $b\in \{b_{j+1},\dots,b_\ell\}$.
    
    Thus, for any $2\le a<b\le |S|$ with $\Sigma(\sigma \circ \pi^*\circ \pi_{b_j,y_j},[a,b])= 0$ we have $b\in \{b_{j+1},\dots,b_\ell\}$. Noting that $\sigma \circ \pi^*\circ \pi_{b_j,y_j}=\sigma \circ \pi_{b_1,y_1}\circ \dots\circ\pi_{b_j,y_j}$, this shows that our chosen $y_j$ satisfies all of the desired conditions.

    After $\ell$ steps, we have constructed distinct $y_1,\dots,y_\ell\in \{1,\dots, |S|\}\setminus \{b_1,\dots,b_\ell\}$ with $b_i<y_i\le b_i+5D$ for $i=1,\dots,\ell$, such that for all $2\le a<b\le |S|$ with $\Sigma(\sigma \circ \pi_{b_1,y_1}\circ \dots\circ\pi_{b_\ell,y_\ell},[a,b])= 0$ we have $b\in \emptyset$. This means that there cannot be any $2\le a<b\le |S|$ with $\Sigma(\sigma \circ \pi_{b_1,y_1}\circ \dots\circ\pi_{b_\ell,y_\ell},[a,b])= 0$. Thus, $\sigma \circ \pi_{b_1,y_1}\circ \dots\circ\pi_{b_\ell,y_\ell}: \{1,\dots, |S|\}\to S$ is a bijection satisfying $\Sigma(\sigma \circ \pi_{b_1,y_1}\circ \dots\circ\pi_{b_\ell,y_\ell},[a,b])\ne 0$ for all $2\le a<b\le |S|$. Thus, defining $s_i=(\sigma \circ \pi_{b_1,y_1}\circ \dots\circ\pi_{b_\ell,y_\ell})(i)$ for $i=1,\dots,|S|$, we have
    \[s_a+\dots+s_b=(\sigma \circ \pi_{b_1,y_1}\circ \dots\circ\pi_{b_\ell,y_\ell})(a)+\dots+(\sigma \circ \pi_{b_1,y_1}\circ \dots\circ\pi_{b_\ell,y_\ell})(b)=\Sigma(\sigma \circ \pi_{b_1,y_1}\circ \dots\circ\pi_{b_\ell,y_\ell},[a,b])\ne 0\]
    for all $2\le a<b\le |S|$. Thus, the partial sums $s_1,s_1+s_2,\dots,s_1+\dots+s_{|S|}$ are all distinct.
\end{proof}

It remains to prove Lemmas \ref{lem:property1-new} to \ref{lem:property4-new}. We start with proving Lemma \ref{lem:property1-new}.

\begin{proof}[Proof of Lemma \ref{lem:property1-new}]
    It suffices to show that for each $b\in \{|S|-30D, \dots, |S|\}$ we have
    \[\mathbb{P}[b\in B(\sigma)]\le \frac{3}{|S|^{\alpha}}.\]
    Indeed, by a union bound this would imply (recalling $(C_\alpha)^\alpha\ge 10^4 \cdot 2^{40D}\ge 10^4D$)
    \[\mathbb{P}[\mathcal{B}_1]\le \sum_{b=|S|-30D}^{|S|}\mathbb{P}[b\in B(\sigma)]\le \frac{(30D+1)\cdot 3}{|S|^{\alpha}}\le \frac{100D}{(C_\alpha)^{\alpha}}\le \frac{1}{100},\]
    as desired. So let $b\in \{|S|-30D, \dots, |S|\}$, then we have
    \[\mathbb{P}[b\in B(\sigma)]\le \sum_{a=2}^{b-1} \mathbb{P}[\Sigma(\sigma,[a,b])=0] =\sum_{a=2}^{b-1} \mathbb{P}[\sigma(a)+\dots+\sigma(b)=0].\]
    For any $2\le a<b\le |S|$, the set $\{\sigma(a),\dots,\sigma(b)\}$ is a uniformly random subset of $S$ of size $b-a+1$. Thus, by Corollary \ref{coro:anticonc-chain} (applied with $k=1$ and $m_1=b-a+1$, noting that $2\le m_1\le |S|-1$), we have 
    \[\mathbb{P}[\sigma(a)+\dots+\sigma(b)=0] \le \Bigg(\frac{1}{p} + \frac{C_1\sqrt{\log |S|}}{|S|\sqrt{b-a+1}}\Bigg)+\Bigg(\frac{1}{p} + \frac{C_1\sqrt{\log |S|}}{|S|\sqrt{|S|-b-1+a}}\Bigg).\]
    Thus, we indeed obtain (recalling that $|S|/p\le |S|^{-\alpha}$ and $4C_1\sqrt{\log |S|}/|S|^{1/2}\le |S|^{-\alpha}$)
    \begin{align*}
    \mathbb{P}[b\in B(\sigma)]&\le \sum_{a=2}^{b-1} \Bigg(\frac{2}{p} + \frac{C_1\sqrt{\log |S|}}{|S|\sqrt{b-a+1}}+\frac{C_1\sqrt{\log |S|}}{|S|\sqrt{|S|-b-1+a}}\Bigg)\\
    &\le \frac{2|S|}{p}+\sum_{i=2}^{b-1} \frac{C_1\sqrt{\log |S|}}{|S|\sqrt{i}}+\sum_{i=|S|-b+1}^{|S|-2} \frac{C_1\sqrt{\log |S|}}{|S|\sqrt{i}}\\
    &\le 2|S|^{-\alpha}+\frac{2C_1\sqrt{\log |S|}}{|S|} \sum_{i=1}^{|S|}\frac{1}{\sqrt{i}}\le 2|S|^{-\alpha}+\frac{4C_1\sqrt{\log |S|}}{|S|^{1/2}}\le 3|S|^{-\alpha}.\qedhere
    \end{align*}
\end{proof}

Before proving Lemmas \ref{lem:property3-new} and \ref{lem:property4-new}, we first prove the following auxiliary lemma.

\begin{lem}\label{lem:property0-new}
    For a uniformly random bijection $\sigma:\{1,\dots, |S|\}\to S$, let $\mathcal{B}_0$ be the event that for some $b\in B(\sigma)$ with  $b\le |S|-30D$ there exist distinct subsets $J,J'\su \{b,b+1,\dots,b+20D\}$ such that $\Sigma(\sigma(J))=\Sigma(\sigma(J'))$. Then we have $\mathbb{P}[\mathcal{B}_0]\le 1/100$.
\end{lem}

\begin{proof}
    It suffices to show that for any $b\in\{1,\dots,|S|-30D\}$ and any two distinct subsets $J,J'\su \{b,b+1,\dots,b+20D\}$ we have
    \begin{equation}\mathbb{P}[b\in B(\sigma)\text{ and }\Sigma(\sigma(J))=\Sigma(\sigma(J'))]\le \frac{8}{|S|^{1+\alpha}}.\label{eq:prob-1}\end{equation}
    Indeed, given (\ref{eq:prob-1}), taking a union bound over the at most $(|S|-30D)\cdot 2^{20D+1}\cdot 2^{20D+1}\le |S|\cdot 2^{40D+2}$ possibilities for $b$, $J$ and $J'$ gives
    \[\mathbb{P}[\mathcal{B}_0]\le |S|\cdot 2^{40D+2}\cdot \frac{8}{|S|^{1+\alpha}}= \frac{2^{40D+5}}{|S|^{\alpha}}\le \frac{2^{40D+5}}{(C_\alpha)^{\alpha}}\le \frac{1}{100}.\]
    So let $b\in\{1,\dots,|S|-30D\}$, and consider distinct subsets $J,J'\su \{b,b+1,\dots,b+20D\}$. Note that the event $\Sigma(\sigma(J))=\Sigma(\sigma(J'))$ (i.e.\ the event $\sum_{j\in J}\sigma(j)=\sum_{j'\in J'}\sigma(j')$) is determined by the outcomes of $\sigma(b),\dots, \sigma(b+20D)$. Furthermore, choosing an index $i\in J\Delta J'$ and conditioning on any outcomes of $\sigma(b),\dots, \sigma(i-1), \sigma(i+1),\dots,\sigma(b+20D)$, we can see that there is at most one outcome for $\sigma(i)$ (among the $|S|-20D$ remaining elements in $S$) such that $\Sigma(\sigma(J))=\Sigma(\sigma(J'))$ is satisfied. Thus, we can conclude
    \[\mathbb{P}[\Sigma(\sigma(J))=\Sigma(\sigma(J'))]\le \frac{1}{|S|-20D
    }\le \frac{2}{|S|},\]
    since $|S|\ge C_\alpha\ge 50D$. Now, conditioning on any fixed outcomes of $\sigma(b),\dots, \sigma(b+20D)$ such that $\Sigma(\sigma(J))=\Sigma(\sigma(J'))$ holds, we have
    \begin{align*}
    &\mathbb{P}[b\in B(\sigma)\mid \sigma(b),\dots, \sigma(b+20D)]\\
    &\qquad\le \sum_{a=2}^{b-1}\mathbb{P}[\Sigma(\sigma,[a,b])=0\mid \sigma(b),\dots, \sigma(b+20D)]\\
    &\qquad=\sum_{a=2}^{b-1}\mathbb{P}[\sigma(a)+\dots+\sigma(b-1)=-\sigma(b)\mid \sigma(b),\dots, \sigma(b+20D)]\\
    &\qquad\le \sum_{a=2}^{b-1} \Bigg(\frac{1}{p} + \frac{C_1\sqrt{\log (|S|-20D-1)}}{(|S|-20D-1)\sqrt{b-a}}+\frac{1}{p} + \frac{C_1\sqrt{\log (|S|-20D-1)}}{(|S|-20D-1)\sqrt{|S|-20D-1-b+a}}\Bigg),
    \end{align*}
    where in the last step we applied Corollary \ref{coro:anticonc-chain} with $k=1$, noting that conditional on our fixed outcomes of $\sigma(b),\dots, \sigma(b+20D)$, the set $\{\sigma(a),\dots,\sigma(b-1)\}$ is a uniformly random subset of size $b-a$ in $S\setminus\{\sigma(b),\dots, \sigma(b+20D)\}$. Therefore we can conclude (recalling that $|S|\ge C_\alpha\ge 50D$ as well as $|S|/p\le |S|^{1-\alpha}$ and $4C_1\sqrt{\log |S|}/|S|^{1/2}\le |S|^{-\alpha}$)
    \[\mathbb{P}[b\in B(\sigma)\mid \sigma(b),\dots, \sigma(b+20D)]\le\frac{2|S|}{p}+2\cdot \frac{C_1\sqrt{\log |S|}}{(|S|/2)}\sum_{i=1}^{|S|}\frac{1}{\sqrt{i}}\le 2|S|^{-\alpha}+\frac{8C_1\sqrt{\log |S|}}{|S|^{1/2}}\le 4|S|^{-\alpha}\]
    for any fixed outcomes of $\sigma(b),\dots, \sigma(b+20D)$. Thus, overall we indeed obtain
    \[\mathbb{P}[b\in B(\sigma)\text{ and }\Sigma(\sigma(J))=\Sigma(\sigma(J'))]\le \frac{2}{|S|}\cdot 4|S|^{-\alpha}=\frac{8}{|S|^{1+\alpha}},\]
    proving \eqref{eq:prob-1}
\end{proof}

Now, we are ready to prove Lemma \ref{lem:property3-new}.

\begin{proof}[Proof of Lemma \ref{lem:property3-new}]
    For the events $\mathcal{B}_0$ and $\mathcal{B}_1$ in Lemmas \ref{lem:property0-new} and \ref{lem:property1-new}, we have $\mathbb{P}[\mathcal{B}_0]\le 1/100$ and $\mathbb{P}[\mathcal{B}_1]\le 1/100$. Therefore it suffices to prove that $\mathbb{P}[\mathcal{B}_2\setminus (\mathcal{B}_0\cup \mathcal{B}_1)]\le 1/100$.

    If $\mathcal{B}_2\setminus (\mathcal{B}_0\cup \mathcal{B}_1)$ (and hence $\mathcal{B}_2$) holds, there is $z\in\{1,\dots,|S|\}$ with $|B(\sigma)\cap \{z-10D,\dots,z+10D\}|> D$. Taking $b_0$ to be the minimal element of $B(\sigma)\cap \{z-10D,\dots,z+10D\}$, this means that there are distinct $b_0,\dots,b_D\in B(\sigma)$ with $b_1,\dots,b_D\in \{b_0+1,\dots,b_0+20D\}$. As $b_1,\dots,b_D\in B(\sigma)$, there are furthermore $a_1,\dots,a_D\in \{2,\dots,|S|\}$ with $a_i<b_i$ and $\Sigma(\sigma,[a_i,b_i])=0$ for $i=1,\dots,D$. Since $\mathcal{B}_1$ does not hold, we must have $b_0\le |S|-30D$.

    We claim that $a_1,\dots,a_D$ must be distinct. Indeed, assume that $a_i=a_h$ for two distinct $i,h\in \{1,\dots,D\}$, and assume without loss of generality that $b_i<b_h$. Now we have
    \[\Sigma(\sigma(\{b_i+1,\dots, b_h\}))=\Sigma(\sigma,[b_i+1,b_h])=\Sigma(\sigma,[a_h,b_h])-\Sigma(\sigma,[a_i,b_i])=0-0=0.\]
    Thus, we have found $b_0\in B(\sigma)$ with $b_0\le |S|-30D$ such that for the subset $J=\{b_i+1,\dots, b_h\}\su \{b_0,\dots, b_0+20D\} $ we have $J\ne \emptyset$, but $\Sigma(\sigma(J))=0=\Sigma(\sigma(\emptyset))$. This means that $\mathcal{B}_0$ holds, which is a contradiction. Therefore $a_1,\dots,a_D$ must indeed be distinct.

    We may now assume without loss of generality that $a_1<\dots<a_D$. We also claim that $a_D<b_0$. Indeed, if $a_D\ge b_0$ we would have $b_0\le a_D<b_D\le b_0+20D$ and furthermore $\Sigma(\sigma(\{a_D,\dots,b_D\}))=\Sigma(\sigma,[a_D,b_D])=0$. Thus, considering $b_0\in B(\sigma)$ and $J=\{a_D,\dots, b_D\}\su \{b_0,\dots, b_0+20D\} $ we would again have $J\ne \emptyset$, but $\Sigma(\sigma(J))=0=\Sigma(\sigma(\emptyset))$. So again $\mathcal{B}_0$ would hold, contradicting our assumption that $\mathcal{B}_2\setminus (\mathcal{B}_0\cup \mathcal{B}_1)$ holds.

    Therefore we must have $a_D<b_0$ and hence $a_1<\dots<a_D<b_0$. Summarizing this, we found that if $\mathcal{B}_2\setminus (\mathcal{B}_0\cup \mathcal{B}_1)$ holds, there must exist $b_0\in \{1,\dots,|S|-30D\}$ and $b_1,\dots,b_D\in \{b_0+1,\dots,b_0+20D\}$ as well as $a_1,\dots,a_D\in \{1,\dots,|S|\}$ with $a_1<\dots<a_D<b_0$ such that $\Sigma(\sigma,[a_i,b_i])=0$ for $i=1,\dots,D$.

    For any given $b_0\in \{1,\dots,|S|-30D\}$ and $b_1,\dots,b_D\in \{b_0+1,\dots,b_0+20D\}$, conditioning on any fixed outcomes of $\sigma(b_0),\dots,\sigma(b_0+20D)$, we have
    \begin{align*}
        &\mathbb{P}\Big[\text{there are } a_1<\dots<a_D<b_0\text{ with }\Sigma(\sigma,[a_i,b_i])=0\text{ for }i=1,\dots,D\,\Big|\,\sigma(b_0),\dots,\sigma(b_0+20D)\Big]\\
        &\qquad\le \sum_{1\le a_1<\dots<a_D<b_0}\mathbb{P}\Big[\Sigma(\sigma,[a_i,b_i])=0\text{ for }i=1,\dots,D\,\Big|\,\sigma(b_0),\dots,\sigma(b_0+20D)\Big]\\
        &\qquad= \sum_{1\le a_1<\dots<a_D<b_0}\mathbb{P}\Big[\Sigma(\sigma,[a_i,b_0-1])=-\Sigma(\sigma,[b_0,b_i])\text{ for }i=1,\dots,D\,\Big|\,\sigma(b_0),\dots,\sigma(b_0+20D)\Big].
    \end{align*}
    Each term on the right-hand side can be bounded by Corollary \ref{coro:anticonc-chain}, noting that $\Sigma(\sigma,[a_i,b_0-1])=\Sigma(\{\sigma(a_i),\dots,\sigma(b_0-1)\})$ for $i=1,\dots,D$, and that $\{\sigma(a_D),\dots,\sigma(b_0-1)\}\su \dots \su \{\sigma(a_1),\dots,\sigma(b_0-1)\}$ (conditional on our fixed outcomes of $\sigma(b_0),\dots,\sigma(b_0+20D)$) is a uniformly random chain of subsets of $S\setminus \{\sigma(b_0),\dots,\sigma(b_0+20D)\}$ of sizes $b_0-a_D,\dots,b_0-a_1$, where $1\le b_0-a_D<\dots<b_0-a_1<|S|-20D-1$. Thus, applying Corollary \ref{coro:anticonc-chain} and Lemma \ref{lem:sum-of-weird-function} we can conclude (also recalling that $|S|/p\le |S|^{-\alpha}$ and $4C_D\sqrt{\log |S|}/|S|^{1/2}\le |S|^{-\alpha}$ as well as $D=\lceil 3/\alpha\rceil$)
    \begin{align*}
        &\mathbb{P}\Big[\text{there are } a_1<\dots<a_D<b_0\text{ with }\Sigma(\sigma,[a_i,b_i])=0\text{ for }i=1,\dots,D\,\Big|\,\sigma(b_0),\dots,\sigma(b_0+20D)\Big]\\
        &\qquad\le (D+1)\cdot \Bigg(\frac{|S|-20D-1}{p} + \frac{2C_D\sqrt{\log (|S|-20D-1)}}{(|S|-20D-1)^{1/2}}\Bigg)^D\\
        &\qquad\le (D+1)\cdot \Bigg(\frac{|S|}{p} + \frac{4C_D\sqrt{\log |S|}}{|S|^{1/2}}\Bigg)^D\le (D+1)\cdot (2|S|^{-\alpha})^D=(D+1)2^D\cdot |S|^{-\alpha D} \le \frac{(D+1)2^D}{|S|^3}
    \end{align*}
    for any given $b_0\in \{1,\dots,|S|-30D\}$ and $b_1,\dots,b_D\in \{b_0+1,\dots,b_0+20D\}$, and any fixed outcomes of $\sigma(b_0),\dots,\sigma(b_0+20D)$. This shows that
    \[\mathbb{P}\Big[\text{there are } a_1<\dots<a_D<b_0\text{ and }\Sigma(\sigma,[a_i,b_i])=0\text{ for }i=1,\dots,D\Big]\le \frac{(D+1)2^D}{|S|^3}\]
    for any $b_0\in \{1,\dots,|S|-30D\}$ and $b_1,\dots,b_D\in \{b_0+1,\dots,b_0+20D\}$. Taking the union bound over all possible choices of $b_0\in \{1,\dots,|S|-30D\}$ and $b_1,\dots,b_D\in \{b_0+1,\dots,b_0+20D\}$, we can conclude that
    \[\mathbb{P}[\mathcal{B}_2\setminus (\mathcal{B}_0\cup \mathcal{B}_1)]\le |S|\cdot (20D)^D\cdot \frac{(D+1)2^D}{|S|^3}=\frac{(D+1)(40D)^D}{|S|^2}\le \frac{(D+1)(40D)^D}{(C_\alpha)^2}\le \frac{1}{100}.\qedhere\]
\end{proof}

To prove Lemma  \ref{lem:property4-new}, we also need another auxiliary lemma.

\begin{lem}\label{lem:auxiliary-lemma-2} 
    Let $b,b'\in \{2,\dots|S|-2\}$ with $b'-b=5D$, let $u_1,\dots,u_D\in \{b,\dots,b'\}$, and consider permutations $\pi_1,\dots,\pi_D: \{1,\dots,|S|\}\to \{1,\dots,|S|\}$ with $\{1,\dots,b-1\}\cup \{b'+1,\dots,|S|\}\su \operatorname{Fix}(\pi_i)$ for $i=1,\dots,D$. Then, for a uniformly random bijection $\sigma:\{1,\dots,|S|\}\to S$, it happens with probability at most $1/|S|^2$ that there exist $x_1,\dots,x_D\in \{b'+1,\dots,|S|\}$ with $x_1<\dots<x_D$, as well as an admissible permutation $\pi:\{1,\dots,|S|\}\to \{1,\dots,|S|\}$, such that $\Sigma(\sigma\circ \pi\circ \pi_i,[u_i,x_i])=0$ for $i=1,\dots,D$.
\end{lem}

By flipping the ordering of $\{1,\dots,|S|\}$, we obtain the following analogous statement from Lemma \ref{lem:auxiliary-lemma-2}.

\begin{lem}\label{lem:auxiliary-lemma-3} 
    Let $b,b'\in \{2,\dots|S|-2\}$ with $b'-b=5D$, let $u_1,\dots,u_D\in \{b,\dots,b'\}$, and consider permutations $\pi_1,\dots,\pi_D: \{1,\dots,|S|\}\to \{1,\dots,|S|\}$ with $\{1,\dots,b-1\}\cup \{b'+1,\dots,|S|\}\su \operatorname{Fix}(\pi_i)$ for $i=1,\dots,D$. Then, for a uniformly random bijection $\sigma:\{1,\dots,|S|\}\to S$, it happens with probability at most $1/|S|^2$ that there exist $x_1,\dots,x_D\in \{1,\dots,b-1\}$ with $x_1<\dots<x_D$, as well as an admissible permutation $\pi:\{1,\dots,|S|\}\to \{1,\dots,|S|\}$, such that $\Sigma(\sigma\circ \pi\circ \pi_i,[x_i,u_i])=0$ for $i=1,\dots,D$.
\end{lem}

Since Lemma \ref{lem:auxiliary-lemma-3} is completely analogous to Lemma \ref{lem:auxiliary-lemma-2}, we only prove Lemma \ref{lem:auxiliary-lemma-2}.

\begin{proof}[Proof of Lemma \ref{lem:auxiliary-lemma-2}]
    For $x_1,\dots,x_D\in \{b'+1,\dots,|S|\}$ with $x_1<\dots<x_D$, let us say that an admissible permutation $\pi:\{1,\dots,|S|\}\to \{1,\dots,|S|\}$ is \emph{interesting} for $(x_1,\dots,x_D)$ if for every pair $(q,r)\in P$ we have $|\{q,r\}\cap \pi_i(\{u_i,\dots,x_i\})|=1$ for some $i\in\{1,\dots,D\}$ (here, by $P$ we denote the admissible collection of pairs corresponding to $\pi$).

    For any outcome of $\sigma:\{1,\dots,|S|\}\to S$, and any $x_1,\dots,x_D\in \{b'+1,\dots,|S|\}$ with $x_1<\dots<x_D$, we claim the following: If there exists an  admissible permutation $\pi:\{1,\dots,|S|\}\to \{1,\dots,|S|\}$ with $\Sigma(\sigma\circ \pi\circ \pi_i,[u_i,x_i])=0$ for $i=1,\dots,D$, then there also exists such an admissible permutation which is interesting for $(x_1,\dots,x_D)$. Indeed, let $P$ be the admissible collection of pairs corresponding to $\pi$ (i.e.\ let $\pi=\pi_P$), and let $P'\su P$ be the subset of those pairs $(q,r)\in P$ satisfying $|\{q,r\}\cap \pi_i(\{u_i,\dots,x_i\})|=1$ for some $i\in\{1,\dots,D\}$. Now, taking $\pi'=\pi_{P'}$ to be the admissible permutation corresponding to $P'$, the permutation $\pi'$ is by definition interesting for $(x_1,\dots,x_D)$. Furthermore, for $i=1,\dots,D$, we have $\pi_{q,r}(\pi_i(\{u_i,\dots,x_i\}))=\pi_i(\{u_i,\dots,x_i\})$ for all $(q,r)\in P\setminus P'$, and hence $\pi_{P\setminus P'}(\pi_i(\{u_i,\dots,x_i\}))=\pi_i(\{u_i,\dots,x_i\})$. This shows that $\pi(\pi_i(\{u_i,\dots,x_i\}))=\pi'(\pi_i(\{u_i,\dots,x_i\}))$ and hence \[\Sigma(\sigma\circ \pi'\circ \pi_i,[u_i,x_i])=\Sigma(\sigma(\pi'(\pi_i(\{u_i,\dots,x_i\}))))=\Sigma(\sigma(\pi(\pi_i(\{u_i,\dots,x_i\}))))=\Sigma(\sigma\circ \pi\circ \pi_i,[u_i,x_i])=0\]
    for $i=1,\dots,D$. Thus, there exists an  admissible permutation $\pi'$, which is interesting for $(x_1,\dots,x_D)$, and satisfies $\Sigma(\sigma\circ \pi'\circ \pi_i,[u_i,x_i])=0$ for $i=1,\dots,D$.

    Thus, the event described in the statement of the lemma is equivalent to the event that there exist $x_1,\dots,x_D\in \{b'+1,\dots,|S|\}$ with $x_1<\dots<x_D$, as well as an admissible permutation $\pi:\{1,\dots,|S|\}\to \{1,\dots,|S|\}$ that is interesting for $(x_1,\dots,x_D)$, such that $\Sigma(\sigma\circ \pi\circ \pi_i,[u_i,x_i])=0$ for $i=1,\dots,D$. So it suffices to show that the probability of this event is at most $1/|S|^2$.

    Let us fix any $x_1,\dots,x_D\in \{b'+1,\dots,|S|\}$ with $x_1<\dots<x_D$. We claim that the number of admissible permutation $\pi:\{1,\dots,|S|\}\to \{1,\dots,|S|\}$ which are interesting for $(x_1,\dots,x_D)$ is at most $D^{14D^2}$. Indeed, letting $P$ denote the admissible collection of pairs corresponding to $\pi$, for every pair $(q,r)\in P$ we have $|\{q,r\}\cap \pi_i(\{u_i,\dots,x_i\})|=1$ for some $i\in\{1,\dots,D\}$. Thus, for every pair $(q,r)\in P$ we have $q\in \{b,\dots,b'\}$ or $r\in \{b,\dots,b'\}$ or $|\{q,r\}\cap \pi_i(\{b'+1,\dots,x_i\})|=1$. Since $\pi_i(\{b'+1,\dots,x_i\})=\{b'+1,\dots,x_i\}$ and $1\le r-q\le 5D=b'-b$, this means that $q\in \{b-5D, \dots,b'\}= \{b-5D, \dots,b+5D\}$ or $q\in \{x_i-5D+1, \dots,x_i\}$ for some $i\in\{1,\dots,D\}$. Thus, we have $q\in \{b-5D,\dots,b+5D\}\cup\bigcup_{i=1}^{D}\{x_i-5D+1, \dots,x_i\}$ for all $(q,r)\in P$. Recall that no two pairs in $P$ can start with the same element $q$, and that $r\in \{q+1,\dots,q+5D\}$ for all $(q,r)\in P$. Thus, to choose $P$, for every element $q\in \{b-5D,\dots,b+5D\}\cup\bigcup_{i=1}^{D}\{x_i-5D+1, \dots,x_i\}$ one has at most $5D$ options of choosing a corresponding element $r$ to take $(q,r)\in P$, in addition to the option of having no pair in $P$ starting with $q$. Noting that $|\{b-5D,\dots,b+5D\}\cup\bigcup_{i=1}^{D}\{x_i-5D+1, \dots,x_i\}|\le 10D+1+D\cdot 5D\le 7D^2$, we can conclude that the number of possibilities for $P$ and hence for $\pi$ is at most $(5D+1)^{7D^2}\le (D^2)^{7D^2}=D^{14D^2}$.

    Now, for any $x_1,\dots,x_D\in \{b'+1,\dots,|S|\}$ with $x_1<\dots<x_D$ and any of the at most $D^{14D^2}$ admissible permutations $\pi:\{1,\dots,|S|\}\to \{1,\dots,|S|\}$ which are interesting for $(x_1,\dots,x_D)$, we have
    \begin{align*}
    &\mathbb{P}[\Sigma(\sigma\circ \pi\circ \pi_i,[u_i,x_i])=0\text{ for }i=1,\dots,D] \\
     &\qquad= \mathbb{P}[\Sigma((\sigma\circ \pi)(\pi_i(\{u_i,\dots,x_i\})))=0\text{ for }i=1,\dots,D]\\
    &\qquad= \mathbb{P}[\Sigma((\sigma\circ \pi)(\{b'+1,\dots,x_i\}))=-\Sigma((\sigma\circ \pi)(\pi_i(\{u_i,\dots,b'\})))\text{ for }i=1,\dots,D],
    \end{align*}
    recalling that $\pi_i(\{b'+1,\dots,x_i\})=\{b'+1,\dots,x_i\}$. Note that, for our fixed $\pi$, the composition $\sigma\circ \pi: \{1,\dots,|S|\}\to S$ is a uniformly random bijection. Conditioning on any fixed outcomes of $(\sigma\circ \pi)(b),\dots,\allowbreak (\sigma\circ \pi)(b')$ (which then determine the value of $\Sigma((\sigma\circ \pi)(\pi_i(\{u_i,\dots,b'\})))$ for all $i=1,\dots,D$), we have
    \begin{align*}
    &\mathbb{P}[\Sigma((\sigma\circ \pi)(\{b'+1,\dots,x_i\}))=\!-\Sigma((\sigma\circ \pi)(\pi_i(\{u_i,\dots,b'\})))\text{ for }i\!=\!1,\dots,D | (\sigma\circ \pi)(b),\dots, (\sigma\circ \pi)(b')]\\
    &\qquad\le \sum_{j=0}^D \prod_{i\in \{0,\dots,D\}\setminus\{j\}} \,\,\Bigg(\frac{1}{p} + \frac{C_D\sqrt{\log s}}{s\sqrt{m_{i+1}(x_1,\dots,x_D)-m_i(x_1,\dots,x_D)}}\Bigg)
    \end{align*}
    by Corollary \ref{coro:anticonc-chain}, noting that $(\sigma\circ \pi)(\{b'+1,\dots,x_1\})\su \dots\su (\sigma\circ \pi)(\{b'+1,\dots,x_D\})$ is a uniformly random chain of subsets of $S\setminus \{(\sigma\circ \pi)(b),\dots, (\sigma\circ \pi)(b')\}$ of sizes $x_1-b',\dots,x_D-b'$, where $1\le x_1-b'<\dots<x_D-b'<|S|-(b'-b+1)$. Here, for convenience we define
    \[s=|S\setminus \{(\sigma\circ \pi)(b),\dots, (\sigma\circ \pi)(b')\}|=|S|-(b'-b+1)=|S|-5D-1,\]
    and $m_i(x_1,\dots,x_D)=x_i-b'$ for $i=1,\dots,D$ as well as $m_0(x_1,\dots,x_D)=0$ and $m_{D+1}(x_1,\dots,x_D)=s$ (then $m_1(x_1,\dots,x_D), \dots, m_D(x_1,\dots,x_D)$ are precisely the subset sizes in the chain).
    
    Since this holds for any fixed outcomes of $(\sigma\circ \pi)(b),\dots, (\sigma\circ \pi)(b')$, we can conclude that
        \begin{align*}
    &\mathbb{P}[\Sigma(\sigma\circ \pi\circ \pi_i,[u_i,x_i])=0\text{ for }i=1,\dots,D] \\
    &\qquad= \mathbb{P}[\Sigma((\sigma\circ \pi)(\{b'+1,\dots,x_i\}))=-\Sigma((\sigma\circ \pi)(\pi_i(\{u_i,\dots,b'\})))\text{ for }i=1,\dots,D]\\
    &\qquad\le \sum_{j=0}^D \prod_{i\in \{0,\dots,D\}\setminus\{j\}} \,\,\Bigg(\frac{1}{p} + \frac{C_D\sqrt{\log s}}{s\sqrt{m_{i+1}(x_1,\dots,x_D)-m_i(x_1,\dots,x_D)}}\Bigg)
    \end{align*}
    for any $x_1,\dots,x_D\in \{b'+1,\dots,|S|\}$ with $x_1<\dots<x_D$ and any of the at most $D^{14D^2}$ admissible permutations $\pi:\{1,\dots,|S|\}\to \{1,\dots,|S|\}$ which are interesting for $(x_1,\dots,x_D)$.

    Thus, for any $x_1,\dots,x_D\in \{b'+1,\dots,|S|\}$ with $x_1<\dots<x_D$, we can conclude
    \begin{align*}
    &\mathbb{P}[\exists \text{ admissible } \pi\text{ s.t. }\Sigma(\sigma\circ \pi\circ \pi_i,[u_i,x_i])=0\text{ for }i=1,\dots,D]\\
    &\qquad=\mathbb{P}[\exists \text{ admissible }\pi,\text{ interesting for }(x_1,\dots,x_D),\text{ s.t. }\Sigma(\sigma\circ \pi\circ \pi_i,[u_i,x_i])=0\text{ for }i=1,\dots,D]\\
    &\qquad \le D^{14D^2}\cdot \sum_{j=0}^D \prod_{i\in \{0,\dots,D\}\setminus\{j\}} \,\,\Bigg(\frac{1}{p} + \frac{C_D\sqrt{\log s}}{s\sqrt{m_{i+1}(x_1,\dots,x_D)-m_i(x_1,\dots,x_D)}}\Bigg).
    \end{align*}
    Thus, the probability that there exist $x_1,\dots,x_D\in \{b'+1,\dots,|S|\}$ with $x_1<\dots<x_D$ and an admissible permutation $\pi:\{1,\dots,|S|\}\to \{1,\dots,|S|\}$, such that $\Sigma(\sigma\circ \pi\circ \pi_i,[u_i,x_i])=0$ for $i=1,\dots,D$, is at most
    \begin{align*}
    &\sum_{b'+1\le x_1<\dots<x_D\le |S|} D^{14D^2}\cdot \sum_{j=0}^D \prod_{i\in \{0,\dots,D\}\setminus\{j\}} \,\,\Bigg(\frac{1}{p} + \frac{C_D\sqrt{\log s}}{s\sqrt{m_{i+1}(x_1,\dots,x_D)-m_i(x_1,\dots,x_D)}}\Bigg)\\
    &\qquad\qquad= D^{14D^2}\sum_{1\le m_1<\dots<m_D\le |S|-b'} \sum_{j=0}^D \prod_{i\in \{0,\dots,D\}\setminus\{j\}} \,\,\Bigg(\frac{1}{p} + \frac{2C_D\sqrt{\log s}}{s\sqrt{m_{i+1}-m_i}}\Bigg)\\
    &\qquad\qquad\le D^{14D^2}\sum_{1\le m_1<\dots<m_D< s} \sum_{j=0}^D \prod_{i\in \{0,\dots,D\}\setminus\{j\}} \,\,\Bigg(\frac{1}{p} + \frac{2C_D\sqrt{\log s}}{s\sqrt{m_{i+1}-m_i}}\Bigg)\\
    &\qquad\qquad\le D^{14D^2}\cdot (D+1)\cdot \Bigg(\frac{s}{p} + \frac{2C_D\sqrt{\log s}}{s^{1/2}}\Bigg)^D \\
    &\qquad\qquad\le D^{14D^2}\cdot (D+1)\cdot \Bigg(\frac{|S|}{p} + \frac{4C_D\sqrt{\log |S|}}{|S|^{1/2}}\Bigg)^D\\
    &\qquad\qquad\le D^{14D^2}\cdot (D+1)\cdot(2|S|^{-\alpha})^D=D^{14D^2}(D+1)2^D\cdot |S|^{-\alpha D}\le \frac{(D+1)2^D D^{14D^2}}{|S|^3}\le \frac{1}{|S|^2}
    \end{align*}
    defining $m_0=0$ and $m_{D+1}=s$. Here, at the first step we used the change of variables $m_i=x_i-b'$ for $i=1,\dots,D$ (noting that then $m_i$ agrees precisely with $m_i(x_1,\dots,x_D)$ defined above), in the third step we used Lemma \ref{lem:sum-of-weird-function}, in the fourth step we used that $|S|/2\le s\le |S|$, in the fifth step we used $|S|/p\le |S|^{-\alpha}$ and $4C_D\sqrt{\log |S|}/|S|^{1/2}\le |S|^{-\alpha}$, in the second-last step $D=\lceil 3/\alpha\rceil$, and in the last step $|S|\ge C_\alpha\ge (D+1)2^D D^{14D^2}$.
    \end{proof}

Finally, we now prove Lemma \ref{lem:property4-new}, which concludes the proof of Theorem \ref{thm:main}.

\begin{proof}[Proof of Lemma \ref{lem:property4-new}]
For the events $\mathcal{B}_0$ and $\mathcal{B}_1$ in Lemmas \ref{lem:property0-new} and \ref{lem:property1-new}, we have $\mathbb{P}[\mathcal{B}_0]\le 1/100$ and $\mathbb{P}[\mathcal{B}_1]\le 1/100$. Therefore it suffices to prove that $\mathbb{P}[\mathcal{B}_3\setminus (\mathcal{B}_0\cup \mathcal{B}_1)]\le 2/100$.

If the event $\mathcal{B}_3\setminus (\mathcal{B}_0\cup \mathcal{B}_1)$  holds, there exists $b\in B(\sigma)$, and admissible permutation $\pi:\{1,\dots, |S|\}\to\{1,\dots, |S|\}$ with $\{1,\dots,b-1\}\su \operatorname{Fix}(\pi)$, and $2D$ different elements $y\in \{b+1,\dots, b+5D\}$ which are blocked for $\sigma$, $b$ and $\pi$. As $\mathcal{B}_1$ does not hold, we must have $b< |S|-30D$ (and we must also have $b\ge 3$). Furthermore, as $\mathcal{B}_0$ does not hold, for any distinct subsets $J,J'\su \{b,b+1,\dots,b+20D\}$ we have $\Sigma(\sigma(J))\ne \Sigma(\sigma(J'))$. Thus, for any distinct subsets $J,J'\su \{b,b+1,\dots,b+5D\}$ we must have $\Sigma(\sigma(\pi(J)))\ne \Sigma(\sigma(\pi(J')))$, since $\pi(J),\pi(J')\su \{b,b+1,\dots,b+10D\}$ are distinct subsets (recalling that $\pi$ is admissible and $\{1,\dots,b-1\}\su \operatorname{Fix}(\pi)$). 

This in particular implies that $\Sigma(\sigma\circ \pi\circ \pi_{b,y}, [s,t])\ne 0$ for any $y\in \{b+1,\dots, b+5D\}$ and any $s,t\in \{b,\dots, b+5D\}$ with $s<t$. Indeed, we have $\Sigma(\sigma\circ \pi\circ \pi_{b,y}, [s,t])=\Sigma(\sigma(\pi(\pi_{b,y}(\{s,\dots,t\}))))\ne \Sigma(\sigma(\pi(\emptyset)))=0$ since $\pi_{b,y}(\{s,\dots,t\})\su \{b,\dots, b+5D\}$ is distinct from $\emptyset$. Therefore, for any $y\in \{b+1,\dots, b+5D\}$ and any interval $[s,t]$ with $\Sigma(\sigma\circ \pi\circ \pi_{b,y}, [s,t])= 0$ we must have $s<b$ or $t>b+5D$.

Recall that an element $y\in \{b+1,\dots, b+5D\}$ is blocked for $\sigma$, $b$ and $\pi$ if we have $\Sigma(\sigma\circ \pi\circ \pi_{b,y}, [s,t])=0$ for some interval $[s,t]$ with $s\in \{b+1,\dots,y\}$ (and hence $t>b+5D$) or $t\in \{b,\dots,y-1\}$ (and hence $s<b$). Since in total at least $2D$ elements $y\in \{b+1,\dots, b+5D\}$ are blocked, one of these two options must happen for at least $D$ different elements $y\in \{b+1,\dots, b+5D\}$.

So let $\mathcal{E}_1$ be the event that there exist $b\in \{2,\dots,|S|-30D\}$ and an admissible permutation $\pi:\{1,\dots, |S|\}\to\{1,\dots, |S|\}$, such that for any distinct subsets $J,J'\su \{b,\dots,b+5D\}$ we have $\Sigma(\sigma(\pi(J)))\ne \Sigma(\sigma(\pi(J')))$, and such that there are distinct $y_1,\dots,y_D\in \{b+1,\dots, b+5D\}$ as well as $s_1,\dots,s_D\in \{b+1,\dots,b+5D\}$ and $t_1,\dots,t_D\in \{b+5D+1,\dots,|S|\}$ with $s_i\le y_i$ such that $\Sigma(\sigma\circ \pi\circ \pi_{b,y_i}, [s_i,t_i])=0$ for $i=1,\dots,D$.

Similarly, let $\mathcal{E}_2$ be the event that there exist $b\in \{2,\dots,|S|-30D\}$ and an admissible permutation $\pi:\{1,\dots, |S|\}\to\{1,\dots, |S|\}$, such that for any distinct subsets $J,J'\su \{b,\dots,b+5D\}$ we have $\Sigma(\sigma(\pi(J)))\ne \Sigma(\sigma(\pi(J')))$, and such that there are distinct $y_1,\dots,y_D\in \{b+1,\dots, b+5D\}$ as well as $s_1,\dots,s_D\in \{1,\dots,b-1\}$ and $t_1,\dots,t_D\in \{b,\dots,b+5D-1\}$ with $t_i< y_i$ and $\Sigma(\sigma\circ \pi\circ \pi_{b,y_i}, [s_i,t_i])=0$ for $i=1,\dots,D$.

Then we have $\mathbb{P}[\mathcal{B}_3\setminus (\mathcal{B}_0\cup \mathcal{B}_1)]\le \mathbb{P}[\mathcal{E}_1]+\mathbb{P}[\mathcal{E}_2]$, and therefore it suffices to prove $\mathbb{P}[\mathcal{E}_1]\le 1/100$ and $\mathbb{P}[\mathcal{E}_2]\le 1/100$.

We start by bounding the probability of the event $\mathcal{E}_1$. For $b,\pi,y_1\dots,y_D,s_1\dots,s_D,t_1\dots,t_D$ as in the definition of the event $\mathcal{E}_1$, we claim that $t_1,\dots,t_D$ must be distinct. Indeed, assume that $t_i=t_j$ for two distinct indices $i,j\in \{1, \dots,D\}$, and assume without loss of generality that $y_i<y_j$. Then we have
\begin{align*}
0&=\Sigma(\sigma\circ \pi\circ \pi_{b,y_i}, [s_i,t_i])\\
&=\Sigma(\sigma(\pi(\pi_{b,y_i}(\{s_i,\dots,t_i\}))))\\
&=\Sigma(\sigma(\pi(\pi_{b,y_i}(\{s_i,\dots,b+5D\}))))+\Sigma(\sigma(\pi(\pi_{b,y_i}(\{b+5D+1,\dots,t_i\}))))\\
&=\Sigma(\sigma(\pi(\pi_{b,y_i}(\{s_i,\dots,b+5D\}))))+\Sigma(\sigma(\pi(\{b+5D,\dots,t_i\}))),
\end{align*}
meaning that $\Sigma(\sigma(\pi(\pi_{b,y_i}(\{s_i,\dots,b+5D\}))))=-\Sigma(\sigma(\pi(\{b+5D,\dots,t_i\})))$. Analogously we have $\Sigma(\sigma(\pi(\pi_{b,y_j}(\{s_j,\dots,b+5D\}))))=-\Sigma(\sigma(\pi(\{b+5D,\dots,t_j\})))$, and hence by our assumption $t_i=t_j$ we can conclude
\[\Sigma(\sigma(\pi(\pi_{b,y_i}(\{s_i,\dots,b+5D\}))))=\Sigma(\sigma(\pi(\pi_{b,y_j}(\{s_j,\dots,b+5D\}))))\]
But this contradicts the conditions in the definition of the event $\mathcal{E}_1$, since $\pi_{b,y_i}(\{s_i,\dots,b+5D\})$ and $\pi_{b,y_j}(\{s_j,\dots,b+5D\})$ are two distinct subsets of $\{b,\dots,b+5D\}$ (since $s_i\le y_i<y_j\le b+5D$, we have $y_j\in \pi_{b,y_i}(\{s_i,\dots,b+5D\})$, but since $s_j>b$, we have $y_j\not\in \pi_{b,y_j}(\{s_j,\dots,b+5D\})$). Thus $t_1,\dots,t_D$ must indeed be distinct. Upon relabeling the indices we may therefore assume $t_1<\dots<t_D$.

Thus, whenever the event $\mathcal{E}_1$ holds, there exist $b\in \{2,\dots,|S|-30D\}$ and $y_1,\dots,y_D,s_1,\dots,s_D\in \{b+1,\dots, b+5D\}$ as well as $t_1,\dots,t_D\in \{b+5D+1,\dots,|S|\}$ with $t_1<\dots<t_D$ and some admissible permutation $\pi:\{1,\dots, |S|\}\to\{1,\dots, |S|\}$, such that $\Sigma(\sigma\circ \pi\circ \pi_{b,y_i}, [s_i,t_i])=0$ for $i=1,\dots,D$. For any given $b\in \{2,\dots,|S|-30D\}$ and $y_1,\dots,y_D,s_1,\dots,s_D\in \{b+1,\dots, b+5D\}$, the probability that this happens is at most $1/|S|^2$ by Lemma \ref{lem:auxiliary-lemma-2} (applied with $b'=b+5D$, as well as $u_i=s_i$ and $\pi_i=\pi_{b,y_i}$ for $i=1,\dots,D$). Thus, by a union bound, we obtain
\[\mathbb{P}[\mathcal{E}_1]\le (|S|-30D-1)\cdot (5D)^{2D}\cdot \frac{1}{|S|^2}\le \frac{(5D)^{2D}}{|S|}\le \frac{(5D)^{2D}}{C_\alpha}\le \frac{1}{100}.\]

It remains to bound the probability of the event $\mathcal{E}_2$. For $b,\pi,y_1\dots,y_D,s_1\dots,s_D,t_1\dots,t_D$ as in the definition of the event $\mathcal{E}_2$, we claim that $s_1,\dots,s_D$ must be distinct. Indeed, assume that $s_i=s_j$ for two distinct indices $i,j\in \{1, \dots,D\}$, and assume without loss of generality that $y_i<y_j$. Then we have
\begin{align*}
0&=\Sigma(\sigma\circ \pi\circ \pi_{b,y_i}, [s_i,t_i])\\
&=\Sigma(\sigma(\pi(\pi_{b,y_i}(\{s_i,\dots,t_i\}))))\\
&=\Sigma(\sigma(\pi(\pi_{b,y_i}(\{s_i,\dots,b-1\}))))+\Sigma(\sigma(\pi(\pi_{b,y_i}(\{b,\dots,t_i\}))))\\
&=\Sigma(\sigma(\pi(\{s_i,\dots,b-1\})))+\Sigma(\sigma(\pi(\pi_{b,y_i}(\{b,\dots,t_i\}))),
\end{align*}
so we obtain that $\Sigma(\sigma(\pi(\pi_{b,y_i}(\{b,\dots,t_i\})))=-\Sigma(\sigma(\pi(\{s_i,\dots,b-1\})))$. Analogously we also have $\Sigma(\sigma(\pi(\pi_{b,y_j}(\{b,\dots,t_j\})))=-\Sigma(\sigma(\pi(\{s_j,\dots,b-1\})))$, and hence by our assumption $s_i=s_j$ we can conclude
\[\Sigma(\sigma(\pi(\pi_{b,y_i}(\{b,\dots,t_i\}))))=\Sigma(\sigma(\pi(\pi_{b,y_j}(\{b,\dots,t_j\}))))\]
But this contradicts the conditions in the definition of $\mathcal{E}_2$, since $\pi_{b,y_i}(\{b,\dots,t_i\})$ and $\pi_{b,y_j}(\{b,\dots,t_j\})$ are two distinct subsets of $\{b,\dots,b+5D\}$ (because we have $y_j=\pi_{b,y_j}(b)\in \pi_{b,y_j}(\{b,\dots,t_j\})$, but $y_j\not\in \pi_{b,y_i}(\{b,\dots,t_i\})$ since $t_i< y_i<y_j$). Thus $s_1,\dots,s_D$ must indeed be distinct. Upon relabeling the indices we may therefore assume $s_1<\dots<s_D$.

Thus, when the event $\mathcal{E}_2$ holds, there are $b\in \{2,\dots,|S|-30D\}$ and $y_1,\dots,y_D\in \{b+1,\dots, b+5D\}$ and $t_1,\dots,t_D\in \{b,\dots, b+5D-1\}$, as well as $s_1,\dots,s_D\in \{1,\dots,b-1\}$ with $s_1<\dots<s_D$ and some admissible permutation $\pi:\{1,\dots, |S|\}\to\{1,\dots, |S|\}$, such that $\Sigma(\sigma\circ \pi\circ \pi_{b,y_i}, [s_i,t_i])=0$ for $i=1,\dots,D$. For any given $b\in \{2,\dots,|S|-30D\}$ and $y_1,\dots,y_D\in \{b+1,\dots, b+5D\}$ and $t_1,\dots,t_D\in \{b,\dots, b+5D-1\}$, the probability that this happens is at most $1/|S|^2$ by Lemma \ref{lem:auxiliary-lemma-3} (applied with $b'=b+5D$, as well as $u_i=t_i$ and $\pi_i=\pi_{b,y_i}$ for $i=1,\dots,D$). Thus, by a union bound, we obtain
\[\mathbb{P}[\mathcal{E}_2]\le (|S|-30D-1)\cdot (5D)^{D}\cdot (5D)^{D}\cdot \frac{1}{|S|^2}\le \frac{(5D)^{2D}}{|S|}\le \frac{(5D)^{2D}}{C_\alpha}\le \frac{1}{100}.\qedhere\]
\end{proof}

\end{document}